\documentclass[10.5pt,a4paper]{article}
\usepackage{graphicx,latexsym,euscript,makeidx,color,bm}
\usepackage{amsmath,amsfonts,amssymb,amsthm,thmtools,mathrsfs,enumerate}
\usepackage[colorlinks,linkcolor=blue,anchorcolor=green,citecolor=red]{hyperref}
\usepackage{graphicx,tikz,exscale,pagecolor}  

\usepackage{booktabs,subfigure,xcolor}

\usepackage{geometry}
 \def\sD{\mathscr{D}}    
\def\dbE{\mathbb{E}}     
\def\dbF{\mathbb{F}}   \def\cF{{\cal F}}  
     
\def\dbH{\mathbb{H}}    \def\BH{{\bm H}}

   \def\cK{{\cal K}}  
     
   \def\cM{{\cal M}}  
   \def\cN{{\cal N}}  
     
\def\dbP{\mathbb{P}}     
     
\def\dbR{\mathbb{R}}     
\def\dbS{\mathbb{S}}     
     
 \def\sU{\mathscr{U}}

 \def\sX{\mathscr{X}}

\def\ss{\smallskip}             \def\hb{\hbox}
\def\ms{\medskip}              \def\ae{\hbox{\rm a.e.}}
        \def\lan{\langle}    \def\as{\hbox{\rm a.s.}}
\def\ds{\displaystyle}   \def\ran{\rangle}    \def\tr{\hbox{\rm tr$\,$}}
         
\def\no{\noindent}          
\def\ns{\noalign{\ss}}     
     \def\essinf{\mathop{\rm essinf}}
\def\rf{\eqref}            
         
\def\deq{\triangleq}     \def\({\Big (}       
\def\les{\leqslant}      \def\){\Big )}       
\def\ges{\geqslant}      \def\[{\Big[}        
\def\ti{\tilde}          \def\]{\Big]}        
\def\wt{\widetilde}      \def\q{\quad}        
\def\h{\widehat}         \def\qq{\qquad}      \def\1n{\negthinspace}
\def\cd{\cdot}           \def\2n{\1n\1n}      
\def\cds{\cdots}         \def\3n{\1n\2n}

\def\a{\alpha}              \def\Om{\Omega}   \def\om{\omega}
         \def\D{\Delta}           
\def\z{\zeta}         \def\Th{\Theta}  \def\th{\theta}    \def\si{\sigma}
\def\e{\varepsilon}     \def\l{\lambda}  \def\m{\mu}      \def\n{\nu}
    \def\t{\tau}     \def\f{\varphi}  \def\i{\infty}   

\def\ba{\begin{array}}                \def\ea{\end{array}}
\def\bel{\begin{equation}\label}      \def\ee{\end{equation}}

\newtheoremstyle{indented}{}{}{\it}{\parindent}{\bfseries}{.}{.5em}{}
\theoremstyle{indented}

\newtheorem{theorem}{Theorem}[section]
\newtheorem{definition}[theorem]{Definition}
\newtheorem{proposition}[theorem]{Proposition}
\newtheorem{corollary}[theorem]{Corollary}
\newtheorem{lemma}[theorem]{Lemma}
\newtheorem{remark}[theorem]{Remark}
\newtheorem{example}[theorem]{Example}
\newtheorem{taggedassumptionx}{}
\newenvironment{taggedassumption}[1]
 {\renewcommand\thetaggedassumptionx{#1}\taggedassumptionx}
 {\endtaggedassumptionx}

\makeatletter
   
   \@addtoreset{equation}{section}
\makeatother
  \allowdisplaybreaks[4]
\def\be{\begin{equation}}
\def\bel{\begin{equation}\label}
\def\ee{\end{equation}}
\def\bea{\begin{eqnarray}}
\def\eea{\end{eqnarray}}
\def\bt{\begin{theorem}\label}
\def\et{\end{theorem}}
\def\bc{\begin{corollary}\label}
\def\ec{\end{corollary}}
\def\bex{\begin{example}\label}
\def\ex{\end{example}}
\def\bl{\begin{lemma}\label}
\def\el{\end{lemma}}
\def\bp{\begin{proposition}\label}
\def\ep{\end{proposition}}
\def\br{\begin{remark}\label}
\def\er{\end{remark}}
\def\ba{\begin{array}}
\def\ea{\end{array}}
\def\bde{\begin{definition}\label}
\def\ede{\end{definition}}

\begin{document}

\title{\bf Time-Inconsistent Stochastic Optimal Control
         Problems and Backward Stochastic Volterra Integral Equations}

\author{Hanxiao Wang\thanks{School of Mathematical Sciences, Fudan University,
                    Shanghai 200433, China (Email: {\tt hxwang14@} {\tt fudan.edu.cn}).}~~~and~~
       Jiongmin Yong\thanks{Department of Mathematics, University of Central Florida,
                           Orlando, FL 32816, USA (Email: {\tt jiongmin.yong@ucf.edu}).
                           This author is supported in part by NSF Grant DMS-1812921.}}

\maketitle

\no\bf Abstract. \rm
An optimal control problem is considered for a stochastic differential equation with the cost functional determined by a backward stochastic Volterra integral equation (BSVIE, for short). This kind of cost functional can cover the general discounting (including exponential and non-exponential) situation with a recursive feature. It is known that such a problem is time-inconsistent in general. Therefore, instead of finding a global optimal control, we look for a time-consistent locally near optimal equilibrium strategy. With the idea of multi-person differential games, a family of approximate equilibrium strategies is constructed associated with partitions of the time intervals. By sending the mesh size of the time interval partition to zero, an equilibrium Hamilton--Jacobi--Bellman (HJB, for short) equation is derived, through which the equilibrium valued function and an equilibrium strategy are obtained. Under certain conditions, a verification theorem is proved and the well-posedness of the equilibrium HJB is established. As a sort of Feynman-Kac formula for the equilibrium HJB equation, a new class of BSVIEs (containing the diagonal value $Z(r,r)$ of $Z(\cd\,,\cd)$) is naturally introduced and the well-posedness of such kind of equations is briefly presented.

\ms

\no\bf Keywords. \rm time-inconsistent optimal control problem, backward stochastic Volterra integral equation,
stochastic differential games, equilibrium strategy, equilibrium Hamilton--Jacobi--Bellman equation.

\ms

\no\bf AMS subject classifications. \rm  93E20, 49N70, 60H20, 35F21, 49L20, 90C39

\section{Introduction}\label{Sec:Introduction}

Let $(\Om,\cF,\dbP)$ be a complete probability space on which a standard $d$-dimensional Brownian motion $W=\{W(t);0\les t<\i\}$ is defined, and let $\dbF=\{\cF_t\}_{t\ges0}$ be the natural filtration of $W(\cd)$ augmented by all the $\dbP$-null sets in $\cF$. Let $T>0$. We denote
\bel{D}\ba{ll}
\ns\ds\sX_t\equiv L_{\cF_{t}}^2(\Om;\dbR^n)=\big\{\xi:\Om\to\dbR^n~| ~\xi \hbox{ is  $\cF_{t}$-measurable, } \dbE|\xi|^2<\i\big\},\q t\in[0,T];  \\
\ns\ds\sD=\big\{(t,\xi)~|~t\in[0,T),\,\xi\in\sX_t\big\};\q\D[0,T]\deq\big\{(t,s)\bigm|0\les t\les s\les T\big\}.\ea\ee
For any {\it initial pair} $(t,\xi)\in\sD$, consider the following controlled (forward) stochastic differential equation (SDE, or FSDE, for short) on the finite horizon $[t,T]$, in its integral form:
\bel{state}X(s)=\xi+\int_t^sb(r,X(r),u(r))dr+\int_t^s\si(r,X(r),u(r))dW(r),\q s\in[t,T],\ee
where $b:[0,T]\times\dbR^n\times U\to\dbR^n$ and $\si:[0,T]\times\dbR^n\times U\to\dbR^{n\times d}$ are given (deterministic) maps with $U\subseteq\dbR^m$ being a nonempty set (which could be bounded or unbounded). In the above, $u(\cd)$ is called a {\it control process}, which belongs to the following set of {\it admissible controls}:
$$\sU[t,T]=\Big\{u:[t,T]\times\Om\to U\bigm|u(\cd)\hb{ is $\dbF$-progressively measurable, and~}\dbE\int^T_t|u(s)|^2ds<\i\Big\}.$$
According to the standard results of SDEs, under appropriate conditions, for any $(t,\xi)\in\sD$ and any $u(\cd)\in\sU[t,T]$, equation \rf{state} admits a unique (strong) solution $X(\cd)\equiv X(\cd\,;t,\xi,u(\cd))$ which is referred to as the {\it state process} corresponding to $(t,\xi)$ and $u(\cd)$. To measure the performance of the control $u(\cd)$, one could introduce the following cost functional:
\bel{cost-e}
J^0(t,\xi;u(\cd))=\dbE_t\[e^{-\l(T-t)}h^0(X(T))+\int_t^T e^{-\l(r-t)}g^0(r,X(r),u(r))dr\],\ee
for some suitable maps $h^0:\dbR^n\to\dbR$, $g^0:[0,T]\times\dbR^n\times U\to\dbR$, and some parameter $\l\ges0$, where $\dbE_t[\,\cd\,]=\dbE[\,\cd\,|\,\cF_t]$ is the conditional expectation operator. The two terms on the right-hand side of the above are called the {\it terminal cost} and the {\it running cost}, respectively, and $\l$ is called the {\it discount rate}. Note that the terms $e^{-\l(T-t)}$ in the terminal cost and $e^{-\l(r-t)}$ in the running cost are exponential functions with the same parameter $\l$. We therefore call \rf{cost-e} a cost functional with an {\it exponential discounting}. With the state equation \rf{state} and the cost functional \rf{cost-e}, one could pose the following classical stochastic optimal control problem.

\ms

{\bf Problem (C)}. For any $(t,\xi)\in\sD$, find a control $\bar u(\cd)\in\sU[t,T]$ such that
\bel{inf1}J^0(t,\xi;\bar u(\cd))=\essinf_{u(\cd)\in\sU[t,T]}J^0(t,\xi;u(\cd))=V^0(t,\xi).\ee

Any $\bar u(\cd)\in\sU[t,T]$ satisfying \rf{inf1} is called an ({\it open-loop}) {\it optimal control} of Problem (C) for the initial pair $(t,\xi)$; the corresponding state process $\bar X(\cd)\equiv X(\cd\,;t,\xi,\bar u(\cd))$ is called an ({\it open-loop}) {\it optimal state process}; $(\bar X(\cd),\bar u(\cd))$ is called an ({\it open-loop}) {\it optimal pair}; and  $V^0(\cd\,,\cd):\sD\to\dbR$ is called the {\it value function} of Problem (C).

\ms

It is known that (see, for  example, \cite{Yong 2012}) if $\bar u(\cd)$ is an open-loop optimal control of Problem (C) for the initial pair $(t,\xi)$ with the corresponding optimal state process
$\bar X(\cd)\equiv\bar X(\cd\,;t,\xi,\bar u(\cd))$, then for any $\t\in(t,T]$,
$$J^0(\t,\bar X(\t);\bar u(\cd)|_{[\t,T]})=\essinf_{u(\cd)\in\sU[\t,T]}J^0(\t,\bar X(\t);u(\cd)).$$
This means that the restriction $\bar u(\cd)|_{[\t,T]}$ of $\bar u(\cd)$ on $[\t,T]$ is an open-loop optimal control for the initial pair $(\t,\bar X(\t))$. Such a property is referred to as the {\it time-consistency} of the optimal control $\bar u(\cd)$,
or the {\it time-consistency} of Problem (C).

\ms

In 1992, Duffie and Epstein introduced {\it stochastic differential utility process} (\cite{Duffie-Epstein 1992,Duffie--Epstein 1992-1}, see also \cite{El Karoui-Peng-Quenez 1997}). We prefer to call it a {\it recursive utility/disutility process}. The main feature of such a process $Y(\cd)$ is that the current value $Y(t)$ depends on the future values $Y(s)$, $t<s\les T$ of the process. Economically, this can be explained as follows: Due to people's optimism/pessimism, the predicted future utility/disutility affects the current utility/disutility. For a classical optimal control problem (with exponential discounting), the recursive utility/disutility process $Y(\cd)$ of the state-control pair $(X(\cd),u(\cd))$ can usually be described by the following equation:
\bel{BSDE1}Y(s)=\dbE_s\[e^{-\l(T-s)}h(X(T))+\int_s^Te^{-\l(r-s)}
g(r,X(r),u(r),Y(r))dr\],\q s\in[t,T].\ee
We refer the readers to \cite{Duffie-Epstein 1992,El Karoui-Peng-Quenez 1997,Wei-Yong-Yu 2017} for more details. It is easy to see that $Y(\cd)$ solves \rf{BSDE1} if and only if for some $Z(\cd)$, the pair $(Y(\cd),Z(\cd))$ is an {\it adapted solution} to the following {\it backward stochastic differential equation} (BSDE, for short):
\bel{BSDE2}Y(s)=h(X(T))+\int_s^T\big[\l Y(r)+g(r,X(r),u(r),Y(r))\big]dr-\int_s^TZ(r)dW(r),\q s\in[t,T].\ee
Inspired by this, we could introduce more general controlled decoupled {\it forward-backward stochastic differential equation} (FBSDE, for short) whose integral form is as follows:
\bel{FBSDE}\left\{\2n\ba{ll}
\ds X(s)=\xi+\int_t^sb(r,X(r),u(r))dr+\int_t^s\si(r,X(r),u(r))dW(r),\\
\ns\ds Y(s)=h(X(T))+\int_s^Tg(r,X(r),u(r),Y(r),Z(r))dr-\int_s^TZ(r) dW(r),\ea\right.\qq s\in[t,T],\ee
with the cost functional, called a {\it recursive cost functional}, given by the following:
\bel{recursive}J^R(t,\xi;u(\cd))=Y(t).\ee
Then we may pose the following classical {\it optimal recursive control problem}.

\ms

{\bf Problem (R).} \rm For each $(t,\xi)\in\sD$, find a $\bar u(\cd)\in\sU[t,T]$ such that
\bel{inf JR}J^R(t,\xi;\bar u(\cd))=\essinf_{u(\cd)\in\sU[t,T]}J^R(t,\xi;u(\cd))\equiv V^R(t,\xi).\ee

Similar to Problem (C), any $\bar u(\cd)\in\sU[t,T]$ satisfies \rf{inf JR} is called an open-loop optimal control, etc. It is interesting that Problem (R) is also time-consistent (see \cite{Wei-Yong-Yu 2017}).

\ms

The advantage of the time-consistency is that for any initial pair $(t,\xi)$, if an optimal control $\bar u(\cd)$ is constructed, then it will stay optimal thereafter (for the later initial pair $(\t,\bar X(\t))$ along the optimal state process). This is a little too ideal. In the real world, the time-consistency rarely exists. Instead, most problems, if not all, people encountered are not time-consistent. In other words, an optimal control/policy found for the current initial pair $(t,\xi)$ will hardly stay optimal as time goes by. Such a situation is referred to as the {\it time-inconsistency}. An important reason leading to the time-inconsistency is due to people's subjective {\it time-preferences}. In fact, people usually over discount on the utility for the outcome of immediate future events. Mathematically, such a situation could be described by the so-called {\it nonexponential discounting}, meaning that the discounting terms $e^{-\l(T-t)}$ and $e^{-\l(s-t)}$ appearing in, say, $J^0(t,\xi;u(\cd))$ are replaced by some more general functions $\mu(T,t)$ and $\nu(s,t)$, respectively, so that the cost functional $J^0(t,\xi;u(\cd))$ becomes
\bel{wt J0}\wt J^0(t,\xi;u(\cd))=\dbE_t\[\m(T,t)h^0(X(T))+\int_t^T\n(r,t)g^0(r,X(r),u(r))dr\].\ee
Let us list some possible nonexponential discounting functions as follows:
\begin{enumerate}[(i)]
\item {\it Hyperbolic discounting }:
 $\mu(t,T)={1\over 1+\l_1(T-t)}$, $\nu(t,r)={1\over1+\l_2(r-t)}$ with $\l_1,\l_2>0$ and they could be equal or different;
\item {\it Heterogeneous discounting}: $\mu(t,T)=e^{-\l_1(T-t)}$, $\nu(t,r)=e^{-\l_2(r-t)}$ with $\l_1,\l_2>0$, $\l_1\ne\l_2$;
\item {\it Convex combination of two exponential discounting }: $\mu(t,T)=\a e^{-\l_1(T-t)}+(1-\a)e^{-\l_2(T-t)} $,
$\nu(t,r)=\a e^{-\l_1(r-t)}+(1-\a)e^{-\l_2(r-t)} $, with $\a\in(0,1)$, $\l_1,\l_2>0$, $\l_1\ne\l_2$;
\item {\it Quasi-exponential discounting}: $\mu(t,T)=\big(1+\a(T-t)\big)e^{-\l(T-t)}$,
$\nu(t,r)=\big(1+\a(r-t)\big)e^{-\l(r-t)}$, with $\a,\l>0$.
\end{enumerate}
We refer the reader to \cite{Ekeland 2006,Ekeland 2008,Ekeland 2010,Marin-Solano 2010,Marin-Solano 2011} for some relevant results. Inspired by \cite{Yong 2012,Yong 2014,Yong 2017,Wei-Yong-Yu 2017}, instead of \rf{wt J0}, one may consider the following more general cost functional
\bel{cost-none}
\h J(t,\xi;u(\cd))=\dbE_t\[h(t,X(T))+\int_t^T g(t,r,X(r),u(r))dr\],
\ee
which not only includes the cases with the discounting functions (i)--(iv) listed above, but also, of course, includes the case of exponential discounting. It turns out that the optimal control problem with the state equation \rf{state} and the above cost functional $\h J(t,\xi;u(\cd))$ is time-inconsistent, in general. If we let
\bel{cost-none-Y}
\h Y(s)=\dbE_s\[h(s,X(T))+\int_s^Tg(s,r,X(r),u(r))dr\],\qq s\in[t,T],\ee
then for some $\h Z(\cd\,,\cd)$, the pair $(\h Y(\cd),\h Z(\cd\,,\cd))$ is the adapted solution to the following {\it backward stochastic Volterra integral equation} (BSVIE, for short):
\bel{BSVIE-noyz}
\h Y(s)=h(s,X(T))+\int_s^T g(s,r,X(r),u(r))dr-\int_s^T\h Z(s,r)dW(r),\q s\in[t,T],\ee
and
$$\h J(t,\xi;u(\cd))=\h Y(t).$$

Motivated by the above non-exponential discounting, together with the recursive utility/disutility, it is then natural to introduce the following recursive cost functional with general (nonexponential) discounting:
\bel{BSVIE}Y(s)=h(s,X(T))+\int_s^T g(s,r,X(r),u(r),Y(r),Z(s,r))dr-\int_s^T Z(s,r)dW(r),\q s\in[t,T],\ee
where  $g:\D[0,T]\times\dbR^n\times U\times\dbR\times\dbR^{1\times d}\to\dbR$ and $h:[0,T]\times\dbR^n\to\dbR$ are suitable deterministic maps, recalling (see \rf{D})
$$\D[0,T]=\{(s,r)\in[0,T]\bigm|0\les s\les r\les T\}.$$
In \rf{BSVIE}, $g(\cd)$ and $h(\cd)$ are called the {\it generator} and the {\it free term}, respectively, of the BSVIE. By the standard results of BSVIE (see, \cite{Yong 2008,Shi-Wang-Yong 2015}), under some mild conditions, for any initial pair $(t,\xi)\in\sD$, any control $u(\cd)\in\sU[t,T]$, and the corresponding state process $X(\cd)$, equation \rf{BSVIE} admits a unique {\it adapted solution} $(Y(\cd),Z(\cd\,,\cd))\equiv(Y(\,\cd\,;t,\xi,u(\cd)),Z(\cd\,,\cd\,
;t,\xi,u(\cd)))$, by which we mean an $(\dbR\times\dbR^{1\times d})$-valued random field $(Y,Z)=\{(Y(s),Z(s,r)):(s,r)\in\D[t,T]\}$ such that
$Y(\cd)$ is $\dbF$-progressively measurable on $[t,T]$; for each fixed $s\in[t,T]$, $Z(s,\cd)$ is $\dbF$-progressively measurable on $[s,T]$; equation \rf{BSVIE} is satisfied in the usual It\^{o} sense. Now, with the adapted solution $(Y(\cd),Z(\cd\,,\cd))$ of BSVIE \rf{BSVIE}, depending on $(t,\xi,u(\cd))$, we introduce the following cost functional:
\bel{cost-functional}J(t,\xi;u(\cd))=Y(t).\ee
Thus, we are considering state equation \rf{state} with the recursive cost functional \rf{cost-functional} determined through BSVIE \rf{BSVIE}. Then the corresponding optimal control problem can be stated as:

\ms

{\bf Problem (N).} For each $(t,\xi)\in\sD$, find a $\bar u(\cd)\in\sU[t,T]$ such that
\bel{Problem-C0}J(t,\xi;\bar u(\cd))=\essinf_{u(\cd)\in\sU[t,T]}J(t,\xi;u(\cd))=V(t,\xi).\ee

Similar to before, any $\bar u(\cd)\in\sU[t,T]$ satisfying \eqref{Problem-C0} is called an ({\it open-loop}) {\it optimal control} of Problem (N) for the initial pair $(t,\xi)$;
the corresponding state process $\bar X(\cd)\equiv X(\cd\,;t,\xi,\bar u(\cd))$ is called an ({\it open-loop}) {\it optimal state process}; $(\bar X(\cd),\bar u(\cd))$ is called an ({\it open-loop}) {\it optimal pair}; and $V(\cd,\cd)$ is called the {\it value function} of Problem (N).

\ms

We point out that Problem (N) is time-inconsistent, in general. Readers might notice that in \cite{Yong 2012,Wei-Yong-Yu 2017}, a similar problem was studied, where the (recursive) cost functional was described by a family of BSDEs. In the last section, we will briefly present an argument to show that using BSVIEs seems to be more natural than using BSDEs. Because of the time-inconsistency of the above Problem (N), people also refer to the above optimal control $\bar u(\cd)$ as the {\it pre-commitment optimal control}.
\ms

Let us take this opportunity to briefly recall the history of BSVIEs. As an extension of BSDEs, a BSVIE of form
\bel{BSVIE1}Y(s)=\xi+\int_s^Tg(s,r,Y(r),Z(s,r))dr-\int_s^TZ(s,r) dW(r),\qq s\in[0,T],\ee
(where $Y(\cd)$ could be higher dimensional) was firstly studied by Lin \cite{Lin 2002}, followed by several other researchers: Aman and N'Zi \cite{Aman-N'Zi 2005}, Wang and Zhang \cite{Wang-Zhang 2007}, Djordjevi\'{c} and Jankovi\'{c} \cite{Djordjevic-Jankovic 2013,Djordjevic-Jankovic 2015}, Hu and {\O}ksendal \cite{Hu 2018}, etc. Inspired by the study of optimal control problems for forward stochastic Volterra integral equations (FSVIEs, for short), Yong \cite{Yong 2008} introduced more general BSVIEs, together with the notion of {\it adapted M-solution}.
There are quite a few follow-up works. Let us mention some: Anh--Grecksch--Yong \cite{Anh-Grecksch-Yong 2011} investigated BSVIEs in Hilbert spaces; Shi--Wang--Yong \cite{Shi-Wang-Yong 2013} studied the well-posedness of BSVIEs containing expectation (of the unknowns); Ren \cite{Ren 2010} discussed BSVIEs with jumps; Wang--Sun--Yong \cite{Wang-Sun-Yong 2018} studied BSVIE with quadratic growth (in $Z$); Wang--Yong \cite{Wang-Yong 2019} obtained a representation of the adapted (M-)solution for a class of BSVIEs via the so-called representation partial differential equations; Overbeck and R\"oder \cite{Overbeck-Roder p} even developed a theory of path-dependent BSVIEs; Numerical aspect was considered by Bender--Pokalyuk \cite{Bender-Pokalyuk 2013}; relevant optimal control problems were studied by Shi--Wang--Yong \cite{Shi-Wang-Yong 2015}, Agram--{\O}ksendal \cite{Agram-Oksendal 2015}, Wang--Zhang \cite{Wang-Zhang 2017}, and Wang \cite{Wang 2018};
Wang--Yong \cite{Wang-Yong 2015} established various comparison theorems for both adapted solutions and adapted M-solutions to BSVIEs in multi-dimensional Euclidean spaces.

\ms

For the state equation \rf{state} together with the recursive cost functional $J(t,\xi;u(\cd))$ defined by \rf{cost-functional} through the BSVIE \rf{BSVIE}, Problem (N) is expected to be  time-inconsistent. Therefore, finding an optimal control at any given initial pair $(t,\xi)$ is not very useful. Instead, one should find an equilibrium strategy which is time-consistent and possesses certain kind of local optimality. To find such a strategy, we adopt the method of multi-person differential games. The idea can be at least traced back to the work of Pollak \cite{Pollak 1968} in 1968. Later, the approach was adopted and further developed by Ekeland--Lazrak \cite{Ekeland 2006,Ekeland 2010}; Yong \cite{Yong 2012,Yong 2014,Yong 2017};
Bj\"{o}rk--Murgoci \cite{Bjork-Murgoci 2014}; Bj\"{o}rk--Murgoci--Zhou \cite{Bjork-Murgoci-Zhou 2014}; Bj\"{o}rk--Khapko--Murgoci \cite{Bjork-Khapko-Murgoci 2017}; Wei--Yong--Yu \cite{Wei-Yong-Yu 2017},  Mei--Yong \cite{Mei-Yong 2019}, and Yan--Yong \cite{Yan-Yong 2019} for various kinds of problems.

\ms

Let us now recall the approach of \cite{Yong 2012}, which leads to our current approach.
For any $\t\in[0,T)$, we first divide the time interval $[\t,T]$ into $N$ subintervals: $[t_0,t_1),[t_1,t_2),...,[t_{N-1},t_N]$ with $t_0=\t$, $t_N=T$,
and introduce an $N$-person differential game, where players are labeled from $1$ to $N$. Player $k$ takes over the system at $t=t_{k-1}$ from Player $(k-1)$, controls the system on $[t_{k-1},t_k)$, and hands over to Player $(k+1)$. The initial pair $(t_{k-1},X(t_{k-1}))$ of Player $k$ is the terminal pair of Player $(k-1)$. All the players know that each player tries to find an optimal control (in some sense) for his/her own problem and each player will discount the future cost in his/her own way, even though this player is not controlling the system later on. For Player $k$, using the representation of adapted solutions to the BSVIE by that of FSDE, together with a presentation partial differential equation, the future cost on $[t_k,T]$ will be transformed to the terminal cost at $t=t_k$; and with some suitable modification of BSVIE on $[t_{k-1},t_k]$, a suitable recursive cost functional on $[t_{k-1},t_k]$ for Player $k$ can be constructed. Then Player $k$ faces a time-consistent optimal control problem. Under proper conditions, an optimal control on $[t_{k-1},t_k]$ can be determined. This leads to an {\it approximate equilibrium strategy} and the corresponding {\it approximate equilibrium value function} of the game. Letting the mesh size $\|\Pi\|$ tend to zero, we get the limits called the {\it equilibrium strategy} and {\it equilibrium value function} of the original problem, respectively. At the same time, a so-called {\it equilibrium Hamilton--Jacobi--Bellman equation} (equilibrium HJB equation, for short) is also derived, which can be used to identify the time-consistent equilibrium value function. Under certain conditions, the equilibrium strategy is locally optimal in a proper sense. When $\si(\cd)$ is independent of the control process $u(\cd)$, under some mild conditions, we will show that the equilibrium HJB equation admits a unique classical solution. Furthermore, inspired by the idea of decoupling FBSDEs and the nonlinear Feynman--Kac formula, the (classical) solution of the equilibrium HJB equation can be represented by the solution of a new type BSVIE in which the term $Z(s,s)$ appears. Under certain conditions, a well-posedness result of such an equation is established. As a consequence, the equilibrium strategy can be expressed in terms of the solution to such a BSVIE.

\ms

The rest of this paper is organized as follows. In \autoref{Preliminaries}, we present preliminary results which will be useful in the sequel. In \autoref{sec:equlibrium-strategy}, by using the idea of multi-person differential games, a family of approximate equilibrium strategies is constructed. By letting the mesh size tend to zero, the equilibrium strategy and equilibrium value function are determined by the equilibrium HJB equation. A verification theorem as well as a well-posedness result of the equilibrium HJB equation (for a special case) is obtained in \autoref{Verification Theorem}. In \autoref{BSVIEs}, a new kind of BSVIE (containing the diagonal values $Z(s,s)$ of $Z(\cd\,,\cd)$) is introduced, which is motivated by finding a Feynman-Kac type formula for the equilibrium HJB equation. Some concluding remarks are collected in \autoref{remarks}, including a formal argument to show that BSVIE is a more suitable way to represent a recursive cost functional with nonexponential discounting, a comparison of equilibrium HJB equations resulted from two approaches, and some open questions concerning Problem (N).

\ms

\section{Preliminaries}\label{Preliminaries}

Throughout this paper, $M^\top$ stands for the transpose of a matrix $M$, $\tr(M)$ the trace of $M$, $\dbR^{n\times d}$ the Euclidean space consisting of $(n\times d)$ real matrices, endowed with the Frobenius inner product $\lan M,N\ran\mapsto\tr[M^\top N]$. Let $U\subseteq \dbR^m$ be a nonempty set which could be bounded or unbounded and $\dbS^n$ be the subspace of $\dbR^{n\times n}$ consisting of symmetric matrices. We will use $K>0$ to represent a generic constant which could be different from line to line. Let $T>0$ be a fixed time horizon and $\dbH$ (also $\dbH_1$, $\dbH_2$) be a Euclidean space (which could be $\dbR^n$, $\dbR^{n\times d}$, $\dbS^n$, etc.). Recall $\D[0,T]=\big\{(t,s)\in[0,T]^2~|~0\les t\les s\les T\big\}$ as before. In the sequel, we will need various spaces of functions and processes, which we collect here first for convenience:
$$\ba{ll}
\ns\ds L_{\cF_t}^2(\Om;\dbH)=\big\{\xi:\Om\to\dbH\bigm|\xi \hbox{ is  $\cF_{t}$-measurable, } \dbE|\xi|^2<\i\big\},\q t\in[0,T],\\
\ns\ds L_\dbF^2(\Om;C([0,T];\dbH))
=\Big\{\f:[0,T]\times\Om\to\dbH\bigm|\f(\cd)~\hb{ $\dbF$-adapted, } t\mapsto\f(t,\om)~\hb{is continuous, }\as\,\om\in\Om,\\
\ns\ds\qq\qq\qq\qq\qq\qq\qq\qq\qq\qq\dbE\big[\ds\sup_{0\les s\les T}|\f(s)|^2\big]<\i \Big\},\\
\ns\ds C_\dbF([0,T];L^2(\Om;\dbH))\1n
=\1n\Big\{\f:[0,T]\1n\to\1n L^2_{\cF_T}(\Om;\dbH)\bigm|\f(\cd)~\hb{is $\dbF$-adapted, continuous, }\ds\sup_{0\les s\les T}\dbE\big[|\f(s)|^2\big]\1n<\1n\i \Big\},\\
\ns\ds L_\dbF^2(\D[0,T];\dbH)
=\Big\{\f:\1n\D[0,T]\1n\times\1n\Om\1n\to\1n\dbH\bigm|\f(t,\cd)~\hb{is $\dbF$-progressively measurable on $[t,T]$, }\ae~t\1n\in\1n[0,T],\\
\ns\ds\qq\qq\qq\qq\qq\qq\qq\qq\qq\qq\sup_{t\in[0,T]}\dbE\int_t^T|\f(t,s)|^2ds<\i \Big\},\\
\ns\ds C^k(\dbH_1;\dbH_2)=\Big\{\f:\1n\dbH_1\to\1n\dbH_2\bigm|\f(\cd)~\hb{is $j$-th continuously differentiable for any $0\les j\les k$}\,\Big\},\\
\ns\ds C_b^k(\dbH_1;\dbH_2)=\Big\{\f:\1n\dbH_1\to\1n\dbH_2\bigm|\f(\cd)\in C^k(\dbH_1;\dbH_2),~\hb{the $j$-th derivatives are bounded, $0\les j\les k$}\,\Big\}.\ea$$

For convenience, we rewrite the state equation and the recursive cost functional below:
\bel{state-rewrite}X(s)=\xi+\int_t^sb(r,X(r),u(r))dr+\int_t^s\si(r,X(r),u(r))dW(r),
\q s\in[t,T],\ee
\bel{BSVIE-rewrite}
Y(s)=h(s,X(T))+\int_s^T g(s,r,X(r),u(r),Y(r),Z(s,r))dr-\int_s^T Z(s,r)dW(r),~s\in[t,T],\ee
and
\bel{cost-functional-rewrite}J(t,\xi;u(\cd))=Y(t).\ee
To guarantee the well-posedness of the controlled SDE \rf{state-rewrite} and BSVIE \rf{BSVIE-rewrite}, we adopt the following assumptions:

\ms

{\bf(H1)} Let the maps $b:[0,T]\times\dbR^n\times U\to\dbR^n$ and $\si:[0,T]\times\dbR^n\times U\to\dbR^{n\times d}$ be continuous. There exists a constant $L>0$ such that
$$\ba{ll}
\ns\ds|b(s,x_1,u)-b(s,x_2,u)|+|\si(s,x_1,u)-\si(s,x_2,u)|\les L|x_1-x_2|,\q \forall~(s,u)\in [0,T]\times U,~ x_1,x_2\in\dbR^n,\\
\ns\ds|b(s,0,u)|+|\si(s,0,u)|\les L(1+|u|),\q\forall~(s,u)\in [0,T]\times U.\ea$$

{\bf(H2)} Let the maps $g:\D[0,T]\times\dbR^n\times U\times\dbR\times\dbR^{1\times d}\to\dbR$ and $h:[0,T]\times\dbR^n\to\dbR$ be continuous. There exists a constant $L>0$ such that
$$\ba{ll}
\ns\ds|g(t_1,s,x_1,u,y_1,z_1)-g(t_2,s,x_2,u,y_2,z_2)|+|h(t_1,x_1)-h(t_2,x_2)|\\
\ns\ds\q\les L\big(|t_1-t_2|+|x_1-x_2|+|y_1-y_2|+|z_1-z_2|\big),\\
\ns\ds\qq\q\forall~ (s,u)\in [0,T]\times U,~ t_1,t_2\in [0,T],~ x_1,x_2\in\dbR^n,~ y_1,y_2\in\dbR,~z_1,z_2\in\dbR^{1\times d},\\
\ns\ds|g(0,s,0,u,0,0)|+|h(0,0)|\les L(1+|u|), \q\forall~ (s,u)\in [0,T]\times U.\ea$$

\ms
The following results, whose proofs are standard (see \cite{Yong-Zhou 1999} and \cite{Yong 2008,Shi-Wang-Yong 2015}), present the well-posedness of SDE \rf{state-rewrite} and BSVIE \rf{BSVIE-rewrite} under (H1)--(H2).

\bl{lmm:well-posedness-SDE}
\sl Let {\rm(H1)} hold. Then for any initial pair $(t,\xi)\in\sD$ and control $u(\cd)\in\sU[t,T]$, the state equation \rf{state-rewrite} admits a unique solution $X(\cd)\equiv X(\cd\,;t,\xi,u(\cd))\in L_\dbF^2(\Om;C([t,T];\dbR^n))$. Moreover, there exists a constant $K>0$, independent of $(t,\xi)$ and $u(\cd)$, such that
\bel{|X|}\dbE_t\(\sup_{t\les s\les T}|X(s)|^2\)\les K\dbE_t\(1+|\xi|^2+ \int_t^T|u(r)|^2dr\).\ee
In addition, if {\rm(H2)} also holds, then for any initial pair $(t,\xi)\in\sD$, control $u(\cd)\in\sU[t,T]$, and the corresponding state process $X(\cd)$, BSVIE \rf{BSVIE-rewrite} admits a unique adapted solution $(Y(\cd),Z(\cd,\cd))\equiv(Y(\cd\,;t,\xi,u(\cd)),Z(\cd,\cd\,;t,\xi,u(\cd)))\in C([t,T];L_\dbF^2(\Om;\dbR))\times L_\dbF^2(\D[t,T];\dbR^{1\times d})$.
Moreover, there exists a constant $K>0$, independent of $(t,\xi)$ and $u(\cd)$, such that
\bel{|Y|+|Z|}\sup_{t\les s\les T}\dbE_t\(|Y(s)|^2+\int_s^T|Z(s,r)|^2dr\)\les K\dbE_t\(1+|\xi|^2+\int_t^T|u(r)|^2dr\).\ee
\el

\ms

We now present a result concerning a certain type modified BSVIEs, which will play a crucial role in our subsequent analysis. For any $(t,\xi)\in\sD$ and $u(\cd)\in\sU[t,T]$, let $(X(\cd),Y(\cd),Z(\cd\,,\cd))$ be the adapted solution to \rf{state-rewrite}--\rf{BSVIE-rewrite}. For any small $\e>0$, consider the following BSVIE:
\bel{BSVIE-e}Y_\e(s)=h_\e(s,X(T))+\int_s^Tg_\e(s,r,X(r),u(r),
Y_\e(r),Z_\e(s,r))dr-\int_s^TZ_\e(s,r)dW(r),\q s\in[t,T],\ee
which is referred to as a {\it modified BSVIE}, where
$$\left\{\1n\ba{ll}
\ds h_\e(s,x)=h(t,x){\bf1}_{[t,t+\e]}(s)+h(s,x){\bf1}_{(t+\e,T]}(s),\\
\ns\ds g_\e(s,r,x,u,y,z)=g(t,r,x,u,y,z){\bf1}_{[t,t+\e]}(s)
+g(s,r,x,u,y,z){\bf1}_{(t+\e,T]}(s).\ea\right.$$
This is a natural approximation of the original BSVIE \rf{BSVIE-rewrite}, in which the generator and the free term have been modified for $s\in[t,t+\e]$ only. By \autoref{lmm:well-posedness-SDE}, under {\rm(H1)--(H2)}, BSVIE \rf{BSVIE-e} admits a unique adapted solution $(Y_\e(\cd),Z_\e(\cd,\cd))$. By the stability of adapted solutions to BSVIEs, we have
\bel{d2}\ba{ll}
\ns\ds\sup_{s\in[t,T]}\dbE_t\(|Y_\e(s)-Y(s)|^2+\int_s^T|Z_\e(s,r)-Z(s,r)|^2dr\)\\
\ns\ds\les K\sup_{s\in[t,T]}\dbE_t\Big[|h_\e(s,X(T))-h(s,X(T))|^2\\
\ns\ds\qq+\int_s^T\2n\big|g_\e(s,r,X(r),u(r),Y(r),Z(s,r))
-g(s,r,X(r),u(r),Y(r),Z(s,r))\big|^2 dr\Big]\\
\ns\ds\les K\Big\{\sup_{s\in[t,t+\e]}\dbE_t\[|h(t,X(T))-h(s,X(T))|^2\\
\ns\ds\qq+\int_s^{T}\2n\big|g(t,r,X(r),u(r),Y(r),Z(s,r))
-g(s,r,X(r),u(r),Y(r),Z(s,r))\big|^2dr\]\Big\}\les K\e^2,\ea\ee
hereafter $K>0$ is a generic constant which could be different from line to line. In particular,
\bel{|Y_e-Y|}|Y_\e(t)-Y(t)|\les K\e.\ee
In our later discussion, we will need a little better than the above. Thus, some more delicate analysis is needed. We now present the following result which gives a representation of $(Y_\e(\cd),Z_\e(\cd\,,\cd))$ in terms of $(X(\cd),Y(\cd),Z(\cd,\cd))$ and a better estimate than \rf{|Y_e-Y|}.

\bp{lemma-BSVIE-approximate-estimate} \sl Let {\rm(H1)}--{\rm(H2)} hold.
Let $(t,\xi)\in\sD$, $u(\cd)\in\sU[t,T]$, and $(X(\cd),Y(\cd),Z(\cd\,,\cd))$ be the adapted solution of \rf{state-rewrite}--\rf{BSVIE-rewrite}. For $\e\in(0,T-t]$, let  $(Y_\e(\cd),Z_\e(\cd\,,\cd))$ be the adapted solution of \rf{BSVIE-e}. Then
\bel{Y_e=Y}\left\{\1n\ba{ll}
\ds Y_\e(s)=Y(s),\qq\qq t+\e<s\les T,\\
\ns\ds Z_\e(s,r)=Z(s,r),\qq t+\e<s\les r\les T,\ea\right.\ee
and
\bel{Z_e=Z}\left\{\1n\ba{ll}
\ds Y_\e(r)=\wt Y(r),\qq\q t\les r\les t+\e,\\
\ns\ds Z_\e(s,r)=\wt Z(r),\qq t\les s\les t+\e,~s\les r\les T,\ea\right.\ee
with $(\wt Y(\cd),\wt Z(\cd))$ being the adapted solution of the following BSDE:
\bel{BSDE-wt Y}\ba{ll}
\ns\ds\wt Y(s)=h(t,X(T))+\int_s^T\(g\big(t,r,X(r),u(r),Y(r),\wt Z(r)\big){\bf1}_{[t+\e,T]}(r)\\
\ns\ds\qq\qq\qq\qq+g\big(t,r,X(r),u(r),\wt Y(r),\wt Z(r)\big){\bf1}_{[t,t+\e]}(r)\)dr-\int_s^T\wt Z(r)dW(r),\q s\in[t,T].\ea\ee
Moreover, there exists a constant $K>0$, independent of $(t,\xi,u(\cd))$, such that
\bel{lemma-BSVIE-approximate-estimate-main}|Y_\e(t)-Y(t)|\les K\e^{2},\qq\as\ee

\ep

\begin{proof} First of all, for $s\in[t+\e,T]$, BSVIEs \rf{BSVIE-e} and \rf{BSVIE-rewrite} are identical. Thus, by the uniqueness of adapted solutions to BSVIEs, we have \rf{Y_e=Y}.
Next, for the given $(X(\cd),u(\cd))$, we denote (suppressing $X(\cd)$ and $u(\cd)$, for notational simplicity)
$$\ba{ll}
\ns\ds h(s)=h(s,X(T)),\qq g(s,r,y,z)=g(s,r,X(r),u(r),y,z),\\
\ns\ds h_\e(s)=h_\e(s,X(T)),\qq g_\e(s,r,y,z)=g_\e(s,r,X(r),u(r),y,z).\ea$$
To get \rf{Z_e=Z}, we let $(y(s,\cd),z(s,\cd))$ and $(y_\e(s,\cd),z_\e(s,\cd))$ be the adapted solutions to the following BSDEs parameterized by $s\in[t,T]$, respectively:
\bel{yz}y(s,\t)=h(s)+\int_\t^Tg(s,r,Y(r),z(s,r))dr-\int_\t^Tz(s,r)
dW(r),\q\t\in[s,T],\ee
and
\bel{y_e}y_\e(s,\t)=h_\e(s)+\int_\t^Tg_\e(s,r,Y_\e(r),
z_\e(s,r))dr-\int_\t^Tz_\e(s,r)dW(r),\q\t\in[s,T].\ee
Note that $Y(\cd)$ and $Y_\e(\cd)$ appeared on the right-hand sides are known. Setting $\t=s$ in \rf{y_e}, we have
\bel{y_e*}y_\e(s,s)=h_\e(s)+\int_s^Tg_\e(s,r,Y_\e(r),z_\e(s,r))dr
-\int_s^Tz_\e(s,r)dW(r),\q s\in[t,T],\ee
Regarding \rf{BSVIE-e} (suppressing $X(\cd)$ and $u(\cd)$) as a BSVIE with generator $(s,r,y,z)\mapsto g(s,r,Y_\e(r),z)$ (independent of $y$), we see that BSVIEs \rf{y_e*} and \rf{BSVIE-e} are identical. Hence, by uniqueness of adapted solutions to BSVIEs, one has
\bel{y_e=Y_e}y_\e(s,s)=Y_\e(s),\qq z_\e(s,r)=Z_\e(s,r),\qq t\les s\les r\les T.\ee
Similarly,
\bel{y=Y}y(s,s)=Y(s),\qq z(s,r)=Z(s,r),\qq t\les s\les r\les T.\ee
Next, for $s\in[t,t+\e]$, \rf{y_e} reads
\bel{wt Y_e*}y_\e(s,\t)=h(t)+\int_\t^Tg(t,r,Y_\e(r),z_\e(s,r))dr
-\int_\t^Tz_\e(s,r)dW(r),\q\t\in[s,T],\ee
which is a BSDE on $[s,T]$ with the generator and terminal term independent of $s\in[t,t+\e]$. Therefore,
\bel{wt Y_e=Y}y_\e(s,\t)=y_\e(\t),\q z_\e(s,\t)=z_\e(\t),\qq s\in[t,t+\e],~\t\in[s,T].\ee
Combining \rf{y_e=Y_e} and \rf{wt Y_e=Y}, we further have that
\bel{Z=Z}z_\e(\t)=z_\e(s,\t)=Z_\e(s,\t)\equiv Z_\e(\t),\qq s\in[t,t+\e],~\t\in[s,T],\ee
i.e., in the current case, $Z_\e(s,\t)$ is independent of $s$ and can be denoted by $Z_\e(\t)$. Clearly, $(y_\e(\cd),z_\e(\cd))$ satisfies (pick $s=t$ in \rf{wt Y_e*}) the following BSDE on $[t,T]$:
\bel{wt Y_e**}y_\e(\t)=h(t)+\int_\t^Tg(t,r,Y_\e(r),z_\e(r))dr-\int_\t^Tz_\e(r)dW(r),\q\t\in[t,T],\ee
Comparing \rf{wt Y_e**} with the following (taking $s=t$ in \rf{yz})
\bel{y(t)}y(t,\t)=h(t)+\int_\t^Tg(t,r,Y(r),z(t,r))dr-\int_\t^Tz(t,r)dW(r),\q
\t\in[t,T],\ee
making use of the fact that (see \rf{Y_e=Y}) $Y_\e(r)=Y(r)$ for $r\in[t+\e,T]$, we see that the above BSDEs \rf{wt Y_e**} and \rf{y(t)} are identical for $\t\in[t+\e,T]$. Hence,
\bel{2.23}y_\e(\t)=y(t,\t)\equiv y(\t),\qq z_\e(\t)=z(t,\t)\equiv z(\t),\qq\t\in[t+\e,T].\ee
Namely, for $\t\in[t+\e,T]$, $(y_\e(\t),z_\e(\t))$ is independent of $\e>0$ and $(y(t,\t),z(t,\t))$ is independent of $t$. Thus, both can be denoted by $(y(\t),z(\t))$. Combining with \rf{Z=Z} and \rf{2.23}, one has
\bel{z=Z_e}z(\t)=Z_\e(\t),\qq\t\in[t+\e,T].\ee
Also, taking $\t=t+\e$ in \rf{wt Y_e**}, we have (noting \rf{2.23})
%
%
%
$$y(t+\e)=h(t)+\int_{t+\e}^Tg(t,r,Y(r),z(r))dr
-\int_{t+\e}^Tz(r)dW(r),$$
which is $\cF_{t+\e}$-measurable, as $(y_\e(\cd),z_\e(\cd))=(y(\cd),z(\cd))$ solves BSDE \rf{wt Y_e**} on $[t+\e,T]$. Hence, (still remember that $s\in[t,t+\e]$ so that \rf{Z=Z}, \rf{2.23} and \rf{z=Z_e} hold) \rf{BSVIE-e} becomes
\bel{BSDE-Y_e}\begin{aligned}
 Y_\e(s)&=h(t)+\int_{t+\e}^Tg(t,r,Y_\e(r),Z_\e(s,r))dr
-\int_{t+\e}^TZ_\e(s,r)dW(r)\\
&\qq\qq+\int_s^{t+\e}g(t,r,Y_\e(r),Z_\e(s,r))dr
-\int_s^{t+\e}Z_\e(s,r)dW(r)\\
&=h(t)+\int_{t+\e}^Tg(t,r,Y(r),z(r))dr
-\int_{t+\e}^Tz(r)dW(r)\\
&\qq\qq+\int_s^{t+\e}g(t,r,Y_\e(r),Z_\e(r))dr
-\int_s^{t+\e}Z_\e(r)dW(r)\\
&=y(t+\e)+\int_s^{t+\e}g(t,r,Y_\e(r),Z_\e(r))dr
-\int_s^{t+\e}Z_\e(r)dW(r),
\end{aligned}
\ee
which is a BSDE on $[t,t+\e]$, since $y(t+\e)$ is $\cF_{t+\e}$-measurable. Also, for fixed $t$, from \rf{yz}, one has the following:
\bel{y(bar t)}y(t,s)=y(t+\e)+\int_s^{t+\e}g(t,r,Y(r),z(t,r))dr
-\int_s^{t+\e}z(t,r)dW(r),\q s\in[t,t+\e].\ee
We now have two BSDEs \rf{BSDE-Y_e} and \rf{y(bar t)} on $[t,t+\e]$ which have the same terminal value $y(t+\e)$, and with different generators $g(t,\cd\,,Y_\e(\cd),z)$ and $g(t,\cd\,,Y(\cd),z)$. Hence, by the stability estimate of BSDE and \rf{d2}, we have
$$\ba{ll}
\ns\ds\dbE_t\[\sup_{s\in[t,t+\e]}|y(t,s)-Y_\e(s)|^2+\int_t^{t+\e}|z(t,r)-Z_\e(r)|^2dr\]\\
\ns\ds\q\les K\dbE_t\(\int_t^{t+\e}|g(t,r,Y(r),Z_\e(r))-g(t,r,Y_\e(r),
Z_\e(r))|dr\)^2\\
\ns\ds\q\les K\dbE_t\(\int_t^{t+\e}|Y(r)-Y_\e(r)|dr\)^2\les K\e\dbE_t\(\int_t^{t+\e}|Y(r)-Y_\e(r)|^2 dr\)\les K\e^4.\ea$$
Consequently,
$$|Y_\e(t)-Y(t)|^2=|Y_\e(t)-y(t,t)|^2\les\sup_{s\in[t,T]}\dbE_t\big[|Y_\e(s)-y(t,s)
|^2\big]\les K\e^4,\qq\as,$$
which proves the estimate \rf{lemma-BSVIE-approximate-estimate-main}. Finally, from \rf{BSDE-Y_e} and \rf{yz}, together with \rf{z=Z_e}, for $s\in[t,t+\e]$, one has
\bel{BSDE-Y_e*}\begin{aligned}
 Y_\e(s)&=h(t)+\int_{t+\e}^Tg(t,r,Y(r),z(r))dr
-\int_{t+\e}^Tz(r)dW(r)\\
&\qq\qq\qq+\int_s^{t+\e}g(t,r,Y_\e(r),Z_\e(r))dr
-\int_s^{t+\e}Z_\e(r)dW(r)\\
&=h(t)+\int_{t+\e}^Tg(t,r,Y(r),Z_\e(r))dr
-\int_{t+\e}^TZ_\e(r)dW(r)\\
&\qq\qq\qq+\int_s^{t+\e}g(t,r,Y_\e(r),Z_\e(r))dr
-\int_s^{t+\e}Z_\e(r)dW(r).\end{aligned}\ee
Hence, if $(\wt Y(\cd),\wt Z(\cd))$ is the adapted solution of \rf{BSDE-wt Y}, then \rf{Z_e=Z} holds. \end{proof}

\ms

Next, let us consider the following system of decoupled FSDEs and BSVIE:
\bel{FBSDE-no-u}\left\{\2n\ba{ll}
\ds X(s)=\xi+\int_t^s b(r,X(r))dr+\int_t^s\si(r,X(r))dW(r),\q s\in[t,T],\\
\ns\ds Y(s)=h(s,X(T))+\int_s^T g(s,r,X(r),Y(r),Z(s,r))dr-\int_s^T Z(s,r)dW(r),\q s\in[t,T],\ea\right.\ee
where $(t,\xi)\in\sD$. For the maps $b,\si,g,h$ in the above, we assume that the corresponding (H1)--(H2) (ignoring $u(\cd)$) hold, by \autoref{lmm:well-posedness-SDE}, system \rf{FBSDE-no-u} admits a unique adapted solution $(X(\cd),Y(\cd),Z(\cd\,,\cd))$. Further, we have the following representation (whose proof can be found in \cite{Wang-Yong 2019}).

\begin{proposition}\label{representation-theorem}
\sl Suppose that $\Th(\cd,\cd,\cd)\in C^{0,1,2}(\D[0,T]\times\dbR^n;\dbR)$ is a classical solution to the following PDE:
\bel{PDE-FBSDE}\left\{\2n\ba{ll}
\ds\Th_s(t,s,x)+\BH^0(t,s,x,\Th(s,s,x),\Th_x(t,s,x),\Th_{xx}(t,s,x))=0,\q (t,s,x)\in\D[0,T]\times\dbR^n,\\
\ns\ds\Th(t,T,x)=h(t,x),\q (t,x)\in[0,T]\times\dbR^n,\ea\right.\ee
where
$$\ba{ll}
\ns\ds\BH^0(t,s,x,\th,p,P)={1\over2}\tr\big[P\si(s,x)\si(s,x)^\top \big]+pb(s,x)+g(t,s,x,\th,p\si(s,x)),\\
\ns\ds\qq\qq\qq\qq\qq\qq(t,s,x,\th,p,P)\in\D[0,T]\times\dbR^n\times\dbR\times\dbR^{1\times n}\times\dbS^n.\ea$$
For any $(t,\xi)\in\sD$, suppose equation \rf{FBSDE-no-u} admits a unique adapted solution $(X(\cd),Y(\cd),Z(\cd\,,\cd))\equiv(X(\cd\,;t,\xi),Y(\cd\,;t,\xi),
Z(\cd\,,\cd\,;t,\xi))$, then
\bel{Y=Th}Y(s)=\Th(s,s,X(s)),~Z(s,r)=\Th_x(s,r,X(r))\si(r,X(r)),\q(s,r)\in\D[t,T],~\as\ee

\end{proposition}

As a convention here and below, for any differentiable function $\f:\dbR^n\to\dbR$ the gradient $\f_x:\dbR^n\to\dbR^{1\times n}$ is a row vector valued function. From the above, we see that when $r\mapsto\Th_x(s,r,x)$ is continuous, $r\mapsto Z(s,r)$ is continuous.

\begin{remark}\rm When
\bel{g0}g(s,r,x,y,z)\equiv g(r,x,y,z),\qq h(s,x)\equiv h(x),\ee
equation \rf{FBSDE-no-u} becomes a decoupled FBSDE. Then \autoref{representation-theorem} turns into a representation theorem of the adapted solution to the corresponding FBSDE.
\end{remark}

To conclude this section, let us recall a result for Problem (R). The corresponding HJB equation reads
\bel{HJB-FBSDE}\left\{\2n\ba{ll}
\ds V^R_s(s,x)+\inf_{u\in U}\BH^R\big(s,x,u,V^R(s,x),V^R_x(s,x),V^R_{xx}(s,x)\big)=0,\q (s,x)\in[0,T]\times\dbR^n,\\
\ns\ds V^R(T,x)=h(x),\q x\in\dbR^n,\ea\right.\ee
with
$$\ba{ll}
\ns\ds\BH^R(s,x,u,y,p,P)={1\over2}\tr\(P\si(s,x,u)\si(s,x,u)^\top\)+pb(s,x,u)+g\big(s,x,u,y,p\si(s,x,u)\big),\\
\ns\ds\qq\qq\qq\qq\qq\qq(s,x,u,y,p,P)\in[0,T]\times\dbR^n\times U\times\dbR\times\dbR^{1\times n}\times\dbS^n.\ea$$

The following result is a standard verification theorem of Problem (R). One is referred to Wei--Yong--Yu \cite{Wei-Yong-Yu 2017} for a proof of this result.

\begin{proposition}\label{veri-thm} \sl Let {\rm(H1)--(H2)} hold with $g$ and $h$ being of form \rf{g0}. Suppose that $V^R(\cd,\cd)\in C^{1,2}([0,T]\times\dbR^n;\dbR)$ is a classical solution of the HJB equation \rf{HJB-FBSDE}. Then for any initial pair $(t,\xi)\in\sD$,
$$V^R(t,\xi)\les J(t,\xi;u(\cd)),\qq\as,\q\forall u(\cd)\in\sU[t,T].$$
Let $(t,\xi)\in\sD$ be a given initial pair and $(\bar X(\cd),\bar u(\cd))$ be the corresponding state-control pair such that
$$\bar u(s)\in\arg\min~\BH^R\big(s,\bar X(s),\cd\,,V^R(s,\bar X(s)),V^R_x(s,\bar X(s)),V^R_{xx}(s,\bar X(s))\big),\q\ae~s\in[t,T],$$
then
$$V^R(t,\xi)= J(t,\xi;\bar u(\cd)),~\as$$
In another word, $V^R(\cd\,,\cd)$ is the value function of  {\rm Problem (R)} and $(\bar X(\cd),\bar u(\cd) )$ is an optimal pair of {\rm Problem (R)} for the initial pair $(t,\xi)$.
\end{proposition}

\section{Feedback Strategy and Equilibrium Strategy}\label{sec:equlibrium-strategy}

\ms

Since Problem (N) is time-inconsistent, instead of looking for an optimal control, we should find a so-called {\it time-consistent equilibrium strategy} for it, which is time-consistent and is locally optimal in a suitable sense. Inspired by the method of multi-person differential games developed in \cite{Yong 2012}, we will introduce {\it approximate equilibrium strategies} associated with partitions of time intervals, together with their constructions, and investigate the limit as the mesh size of the partition tends to zero. To this end, let us first introduce the following definition.

\bde{feedback strategy} \rm Let $\t\in[0,T)$. A map $\Psi:[\t,T]\times\dbR^n\to U$ is called a {\it feedback strategy} (of state equation \rf{state}) on $[\t,T]$ if for every $t\in[\t,T]$ and $\xi\in\sX_t$, the following {\it closed-loop} system:
\bel{closed-loop}X(s)=\xi+\int_t^sb\big(s,X(s),\Psi(s,X(s))\big)ds+\int_t^s\si\big(s,X(s),
\Psi(s,X(s))\big)dW(s),\q s\in[t,T],\ee
admits a unique solution $X(\cd)\equiv X(\cd\,;t,\xi,\Psi(\cd\,,\cd))\equiv X^\Psi(\cd)$.

\ede

Under (H2), for a given feedback strategy $\Psi(\cd\,,\cd)$, the following BSVIE:
\bel{closed-cost}\ba{ll}
\ns\ds Y(s)=h(s,X^\Psi(T))+\int_s^Tg(s,r,X^\Psi(r),\Psi(r,X^\Psi(r)),Y(r),Z(s,r))ds\\
\ns\ds\qq\qq\qq\qq\qq-\int_s^TZ(s,r)dW(r),\q s\in[t,T]\ea\ee
admits a unique adapted solution $(Y(\cd),Z(\cd\,,\cd))\equiv(Y^\Psi(\cd),Z^\Psi(\cd\,,\cd))$. Therefore, the corresponding recursive cost functional at $(t,\xi)$ is well-defined:
\bel{cost*}J(t,\xi;\Psi(\cd\,,\cd))=Y^\Psi(t).\ee
In this case, the outcome $u(\cd)$ of $\Psi(\cd\,,\cd)$, called a {\it closed-loop control}, given by the following
$$u(s)=\Psi(s,X^\Psi(s)),\qq s\in[t,T],$$
is {\it time-consistent}. Moreover, by  \autoref{representation-theorem}, if $\Th^\Psi(\cd,\cd,\cd)\in C^{0,1,2}(\D[\t,T]\times\dbR^n;\dbR)$ is a classical solution to the following representation PDE (parameterized by $t\in[\t,T]$):
\bel{PDE-FBSDE*}\left\{\2n\ba{ll}
\ds\Th^\Psi_s(t,s,x)+\BH^\Psi\big(t,s,x,\Th^\Psi(s,s,x),\Th^\Psi_x
(t,s,x),\Th^\Psi_{xx}(t,s,x)\big)=0,\q(t,s,x)\in\D[\t,T]\times\dbR^n,\\
\ns\ds\Th^\Psi(t,T,x)=h(t,x),\q(t,x)\in[\t,T]\times\dbR^n,\ea\right.\ee
where
$$\ba{ll}
\ns\ds\BH^\Psi(t,s,x,\th,p,P)={1\over2}\tr\big[P\si\big(s,x,\Psi(s,x)\big)
\si\big(s,x,\Psi(s,x)\big)^\top\big]+pb\big(s,x,\Psi(s,x)\big)\\
\ns\ds\qq\qq\qq\qq\q+g\big(t,s,x,\Psi(s,x),\th,p\si(s,x,\Psi(s,x))\big),\\
\ns\ds\qq\qq\qq\qq\qq\qq(t,s,x,\th,p,P)\in\D[\t,T]\times\dbR^n\times\dbR\times\dbR^{1\times n}\times\dbS^n,\ea$$
then $(Y^\Psi(\cd),Z^\Psi(\cd\,,\cd))$ admits the following representation:
\bel{Y=Th*}\left\{\2n\ba{ll}
\ds Y^\Psi(s)=\Th^\Psi(s,s,X^\Psi(s)),\qq s\in[t,T],~\as,\\
\ns\ds Z^\Psi(s,r)=\Th^\Psi_x(s,r,X^\Psi(r))\si\big(r,X^\Psi(r),\Psi(r,X^\Psi(r))\big),
\q(s,r)\in\D[t,T],~\as,\ea\right.\ee
with $X^\Psi(\cd)$ being the solution of \rf{closed-loop}.

\ms

Next, we introduce the following notion, which combines the time-consistency and the local optimality.

\bde{equilibrium strategy} \rm A feedback strategy $\Psi:[\t,T]\times\dbR^n\to U$ (of system \rf{state}) is called an {\it equilibrium strategy} on $[\t,T]$ if for any $t\in[\t,T)$, $\xi\in\sX_t$, and $\e>0$ small with $t+\e\les T$, there exists a family of feedback strategies $\Psi^\e:[t,T]\times\dbR^n\to U$ such that
\bel{lim}\left\{\2n\ba{ll}
\ds\Psi^\e(s,x)=\Psi(s,x),\qq\qq(s,x)\in[t+\e,T]\times\dbR^n,\\
\ns\ds\lim_{\e\to0}\sup_{s\in[t,t+\e),|x|\les M}|\Psi^\e(s,x)-\Psi(s,x)|=0,\qq\forall M>0,\ea\right.\ee
and
\bel{optimality}J(t,\xi;\Psi^\e)\les J\big(t,\xi;u\oplus\Psi|_{[t+\e,T]}\big)+o(\e),\qq\forall u(\cd)\in\sU[t,t+\e],\ee
where $o(\e)$ is uniform in $u(\cd)\in\sU[t,t+\e]$ and
\bel{u+Psi}\big(u\oplus\Psi|_{[t+\e,T]}\big)(s)=\left\{\2n\ba{ll}
\ds u(s),\qq\qq\q\,\, s\in[t,t+\e),\\
\ns\ds\Psi(s,X^\Psi(s)),\qq s\in[t+\e,T].\ea\right.\ee
\ede

From \rf{optimality}, we see that the family $\{\Psi^\e(\cd\,,\cd)\}_{\e>0}$ satisfies the following:
\bel{near-optim}J(t,\xi;\Psi^\e)\les\inf_{u(\cd)\in\sU[t,t+\e]}J\big(t,\xi;u\oplus
\Psi|_{[t+\e,T]}\big)+o(\e),\ee
which can be referred to as the {\it local near-optimality} of the family $\{\Psi^\e(\cd\,,\cd)\}_{\e>0}$ at $(t,\xi)$ (see \cite{Zhou-1998} for the notion of near-optimality, for standard time-consistent problems). Further, if the following holds
$$\sup_{s\in[t,t+\e),|x|\les M}|\Psi^\e(t,x)-\Psi(t,x)|=o(\e),\qq\forall M>0,$$
then \rf{near-optim} leads to
\bel{near-optim*}J\big(t,\xi;\Psi\big|_{[t,T]}\big)\les\inf_{u(\cd)\in\sU[t,t+\e]}
J\big(t,\xi;u\oplus\Psi|_{[t+\e,T]}\big)+o(\e),\ee
which can be referred to as the local near-optimality of $\Psi(\cd\,,\cd)$ itself (instead of the family $\{\Psi^\e(\cd\,,\cd)\}_{\e>0}$).

\ms

To construct equilibrium strategies, we need to make some preparations. First, we set
\bel{Hamilton}\ba{ll}
\ns\ds\BH(t,s,x,u,\th,p,P)={1\over2}\tr\[P\si(s,x,u)\si(s,x,u)^\top\]+pb(s,x,u)
+g\big(t,s,x,u,\th,p\si(s,x,u)\big),\\
\ns\ds\qq\qq\qq\qq\qq\qq\qq\qq(t,s,x,u,\th,p,P)\in\D[0,T]\times\dbR^n\times U\times\dbR\times\dbR^{1\times n}\times\dbS^n.\ea\ee
The infimum (or minimum) of the map $u\mapsto\BH(t,s,x,u,\th,p,P)$ will be needed below. However, this map may not be bounded below in general; even if it is bounded below, the infimum might not be achieved; and even if the minimum exists (in the case, say, $U$ is compact), the minimum might not be unique and might not have needed regularity properties. To avoid all these inconvenient situations which are not our main concern in this paper, similar to \cite{Yong 2012}, we introduce the following technical assumption, which, as mentioned in \cite{Yong 2012}, is satisfied by some situations.

\ms

{\bf(H3)} There is a continuous map $\psi:\D[0,T]\times\dbR^n\times\dbR\times\dbR^{1\times n}\times\dbS^n\to U$ such that
\bel{define-psi}\ba{ll}
\ns\ds\BH\big(t,s,x,\psi(t,s,x,\th,p,P),\th,p,P\big)
=\min_{u\in U}\BH(t,s,x,u,\th,p,P),\\
\ns\ds\qq\qq\qq\qq\qq(t,s,x,\th,p,P)\in \D[0,T]\times\dbR^n\times\dbR\times\dbR^{1\times n}\times\dbS^n,\ea\ee
with the properties that the map
$$(t,s,x,\th,p,P)\mapsto\BH\big(t,s,x,\psi(t,s,x,\th,p,P),\th,p,P\big)$$
is continuously differentiable having bounded derivatives in its arguments.

\ms

We admit that the above assumption is restrictive and maybe far more than enough. The problem without the above (H3) is widely open. We prefer not to explore the minimization problem in (H3) here, and leave it to our future investigations.

\ms



%




Now, let $0\les\t<\bar\t<T$ and $\Psi:[\bar\t,T]\times\dbR^n\to U$ be a feedback strategy on $[\bar\t,T]$. We want to extend $\Psi(\cd\,,\cd)$ to $[\t,T]\times\dbR^n$ by means of optimal controls. To this end, we formulate a time-consistent optimal control problem on $[\t,\bar \t]$. For any $t\in[\t,\bar\t]$, $\xi\in\sX_t$, and $u(\cd)\in\sU[t,\bar\t]$, we apply $u(\cd)$ to the system on $[t,\bar\t)$ followed by the feedback strategy $\Psi(\cd\,,\cd)$ on $[\bar\t,T]$. Then the state equation reads (recall \rf{u+Psi})
\bel{state[u+Psi]}\ba{ll}
\ns\ds X(s)=\xi+\int_t^sb\big(r,X(r),(u\oplus\Psi|_{[\bar\t,T]})(r,X(r))\big)dr\\
\ns\ds\qq\qq\qq+\int_t^s\si\big(r,X(r),(u\oplus\Psi|_{[\bar\t,T]})(r,X(r))\big)dW(r),\qq s\in[t,T].\ea\ee
The corresponding recursive cost functional (on $[t,T]$) should be
\bel{cost[u+Psi]}J\big(t,\xi;(u\oplus\Psi|_{[\bar\t,T]})(\cd)\big)=Y(t),\ee
where $(Y(\cd),Z(\cd\,,\cd))$ is the adapted solution to the following BSVIE:
\bel{BSVIE[u+Psi]}\ba{ll}
\ns\ds Y(s)=h(s,X(T))+\int_s^Tg\big(s,r,X(r),(u\oplus\Psi|_{[\bar\t,T]})(r,X(r)),Y(r),Z(s,r)\big)dr\\
\ns\ds\qq\qq\qq\qq\qq-\int_s^TZ(s,r)dW(r),\qq s\in[t,T].\ea\ee
Note that for $s\in[t,\bar\t]$, the state $X(\cd)$ satisfies the following:
\bel{state[u]}X(s)=\xi+\int_t^sb(r,X(r),u(r))dr+\int_t^s\si(r,X(r),u(r)
)dW(r),\q s\in[t,\bar\t],\ee
and the BSVIE \rf{BSVIE[u+Psi]} can be written as
\bel{BSVIE[t,t]}\ba{ll}
\ns\ds Y(s)=h(s,X(T))+\int_{\bar\t}^Tg\big(s,r,X(r),\Psi(r,X(r)),Y(r),Z(s,r)\big)dr-\int_{\bar\t}^T
Z(s,r)dW(r)\\
\ns\ds\qq\qq+\int_s^{\bar\t}g(s,r,X(r),u(r),Y(r),Z(s,r))dr-\int_s^{\bar\t}Z(s,r)dW(r),\qq s\in[t,\bar\t].\ea\ee
It seems to be difficult to write the above \rf{BSVIE[t,t]} as a BSDE on $[t,\bar\t]$ in general. In fact, even if the sum of the first three terms on the right-hand side is $\cF_{\bar\t}\,$-measurable, due to the dependence on $s$ of this sum and the integrand of the fourth term, one at most can get a BSVIE on $[t,\bar\t]$. Consequently, the optimal control problem associated with the state equation \rf{state[u]} and cost functional \rf{cost[u+Psi]} (determined through \rf{BSVIE[t,t]}) is generally time-inconsistent, which could not be handled by the classical dynamic programming approach. Therefore, instead of \rf{BSVIE[u+Psi]}, we introduce the following modified BSVIE:
\bel{BSVIE[u+Psi]m}\ba{ll}
\ns\ds\wt Y(s)=h_\t(s,X(T))+\int_s^Tg_\t\big(s,r,X(r),(u\oplus\Psi|_{[\bar\t,T]})(r,X(r)),\wt Y(r),\wt Z(s,r)\big)dr\\
\ns\ds\qq\qq\qq\qq\qq\qq-\int_s^T\wt Z(s,r)dW(r),\qq s\in[t,T],\ea\ee
where
$$\left\{\1n\ba{ll}
\ds h_\t(s,x)=h(\t,x){\bf1}_{[\t,\bar\t]}(s)+h(s,x){\bf1}_{(\bar\t,T]}(s),\\
\ns\ds g_\t(s,r,x,u,y,z)=g(\t,r,x,u,y,z){\bf1}_{[\t,\bar\t]}(s)
+g(s,r,x,u,y,z){\bf1}_{(\bar\t,T]}(s).\ea\right.$$
Then we define the recursive cost functional for $(t,\xi)\in[\t,\bar\t]\times\sX_t$ as follows:
\bel{J^t}J^\t(t,\xi;u(\cd))=\wt Y(t),\ee
and pose the following optimal control problem.

\ms

\bf Problem (C$^\Psi[\t,\bar\t]$). \rm For given $(t,\xi)\in[\t,\bar\t]\times\sX_t$, find a control $\bar u(\cd)\in\sU[t,\bar\t]$ such that
$$J^\t(t,\xi;\bar u(\cd))=\inf_{u(\cd)\in\sU[t,\bar\t]}J^\t(t,\xi;u(\cd))\equiv V^\t(t,\xi).$$

\ms

The above is referred to as a {\it sophisticated} optimal control problem on $[\t,\bar\t]$. We have the following result concerning the above problem.

\bp{time-consistent} \sl Let $\Th^\Psi(\cd\,,\cd\,,\cd)$ be a classical solution to the following:
\bel{PDE-Th^Psi}\left\{\2n\ba{ll}
\ds\Th^\Psi_s(t,s,x)+\BH\big(t,s,x,\Psi(s,x),\Th^\Psi(s,s,x),\Th^\Psi_x(t,s,x),\Th^\Psi_{xx}
(t,s,x)\big)=0,\\
\ns\ds\qq\qq\qq\qq\qq\qq\qq\qq\qq(t,s,x)\in\D[\bar\t,T]\times\dbR^n,\\
\ns\ds\Th^\Psi(t,T,x)=h(t,x),\qq (t,x)\in[\bar\t,T]\times\dbR^n,\ea\right.\ee
and let $\wt\Th^\Psi(\t,\cd\,,\cd)$ be a classical solution to the following:
\bel{PDE-Th^Psi*}\left\{\2n\ba{ll}
\ds\wt\Th^\Psi_s(\t,s,x)+\BH\big(\t,s,x,\Psi(s,x),\Th^\Psi(s,s,x),\wt\Th^\Psi_x(\t,s,x),
\wt\Th^\Psi_{xx}(\t,s,x)\big)=0,\\
\ns\ds\qq\qq\qq\qq\qq\qq\qq\qq\qq\qq\qq(s,x)\in[\bar\t,T]\times\dbR^n,\\
\ns\ds\wt\Th^\Psi(\t,T,x)=h(\t,x),\q x\in\dbR^n.\ea\right.\ee
Then
\bel{BSDE-wt Y*}\ba{ll}
\ns\ds\wt Y(s)=\wt\Th^\Psi(\t,\bar\t,X(\bar\t))+\int_s^{\bar\t}g\big(\t,r,X(r),u(r),\wt Y(r),
\wt Z(r)\big)dr-\int_s^{\bar\t}\wt Z(r)dW(r),\q s\in[\t,\bar\t].\ea\ee
Consequently, Problem {\rm(C$^\Psi[\t,\bar\t]$)} is time-consistent on $[\t,\bar\t]$.

\ep

\begin{proof} \rm By  \autoref{lemma-BSVIE-approximate-estimate}, we have
\bel{Y_e=Y*}\left\{\1n\ba{ll}
\ds\wt Y(s)=Y(s),\qq\qq\bar\t\les s\les T,\\
\ns\ds\wt Z(s,r)=Z(s,r),\qq\bar\t\les s\les r\les T,\ea\right.\ee
and
\bel{Z_e=Z*}\wt Z(s,r)=\wt Z(r),\qq t\les s\les\bar\t,~s\les r\les T,\ee
with $(\wt Y(\cd),\wt Z(\cd))$ being the adapted solution of the following BSDE:
\bel{BSDE-wt Y*1}\ba{ll}
\ns\ds\wt Y(s)=h(\t,X(T))+\int_s^T\(g\big(\t,r,X(r),\Psi(r,X(r)),Y(r),\wt Z(r)\big){\bf1}_{[\bar\t,T]}(r)\\
\ns\ds\qq\qq\qq+g\big(\t,r,X(r),u(r),\wt Y(r),\wt Z(r)\big){\bf1}_{[\t,\bar\t]}(r)\)dr-\int_s^T\wt Z(r)dW(r),\q s\in[t,\bar\t].\ea\ee
On the other hand, by  \autoref{representation-theorem}, we have the following representation:
\bel{Y=Th*1}Y(s)=\Th^\Psi(s,s,X(s)),~Z(s,r)=\Th^\Psi_x\big(s,r,X(r)\big)\si\big(r,X(r),
\Psi(r,X(r))\big),\q(s,r)\in\D[\bar\t,T],~\as,\ee
where $\Th^\Psi(\cd\,,\cd\,,\cd)$ is the solution to \rf{PDE-Th^Psi}. Thus, on $[\bar\t,T]$, we have the following FBSDE:
\bel{FBSDE[t,T]}\left\{\2n\ba{ll}
\ds X(s)=X(\bar\t\,)+\int_{\bar\t}^sb\big(r,X(r),\Psi(r,X(r))\big)dr+\int_{\bar \t}^s\si\big(r,X(r),\Psi(r,X(r))\big)dW(r),\\
\ns\ds\wt Y(s)=h(\t,X(T))+\int_s^Tg\big(\t,r,X(r),\Psi(r,X(r)),\Th^\Psi(r,r,X(r)),\wt Z(r)\big)dr-\int_s^T\wt Z(r)dW(r).\ea\right.\ee
Then
\bel{}\wt Y(s)=\wt\Th^\Psi(\t,s,X(s)),\q\wt Z(s)=\wt\Th^\Psi_x(\t,s,X(s))\si\big(s,X(s),\Psi(s,X(s))\big),\qq s\in[\bar\t,T],\ee
with $\wt\Th^\Psi(\t,\cd\,,\cd)$ solving \rf{PDE-Th^Psi*}. Hence, in particular,
$$\wt Y(\bar\t)=\wt\Th^\Psi(\t,\bar\t,X(\bar\t)).$$
Then \rf{BSDE-wt Y*} follows, which is a BSDE on $[\t,\bar\t]$. This leads to that Problem (C$^\Psi[\t,\bar\t]$) is a standard optimal control problem with a recursive cost functional. Hence, it is time-consistent on $[\t,\bar\t]$. \end{proof}

Let us make a comment on \rf{PDE-Th^Psi} and \rf{PDE-Th^Psi*}. In the former, $t$ is not fixed since in the equation, both $\Th^\Psi(t,s,x)$ and $\Th^\Psi(s,s,x)$ appear. Whereas, in the latter, $\t$ only plays a role of parameter and it is fixed. Clearly, the former is much difficult than the latter, and their structures are essentially different.

\ms

From \rf{lemma-BSVIE-approximate-estimate-main}, we know that there exists a constant $K>0$, independent of $(t,\xi,u(\cd))$ such that
\bel{|Y-Y|}|\wt Y(t)-Y(t)|\les K(\bar\t-t)^{2},\qq\as,~t\in[\t,\bar\t].\ee
Now, since Problem (C$^\Psi[\t,\bar\t]$) is a standard optimal control problem with a recursive cost functional which is time-consistent, we may use dynamic programming method. Thus, the value function $V^\t(\cd\,,\cd)$ of Problem (C$^\Psi[\t,\bar\t]$) is the unique viscosity solution to the following HJB equation:
\bel{HJB(t)}\left\{\2n\ba{ll}
\ds V^\t_t(t,x)+\inf_{u\in U}\BH\big(\t,t,x,u,V^\t(t,x),V^\t_x(t,x),V^\t_{xx}(t,x)\big)=0,\q(t,x)\in[\t,\bar\t]
\times\dbR^n,\\
\ns\ds V^\t(\bar\t,x)=\Th(\t,\bar\t,x),\qq x\in\dbR^n,\ea\right.\ee
with the Hamiltonian $\BH$ given by \rf{Hamilton}. Further, under the non-degenerate condition, $V^\t(\cd\,,\cd)$ is the unique classical solution to the above HJB equation. Then, we may define
\bel{Psi*}\Psi(t,x)=\psi\big(\t,t,x,V^\t(t,x),V^\t_x(t,x),V^\t_{xx}(t,x)\big),
\qq(t,x)\in[\t,\bar\t)\times\dbR^n.\ee
This extends $\Psi$ from $[\bar\t,T]\times\dbR^n$ to $[\t,T]\times\dbR^n$. For convenience, we refer to the above as the {\it strategy extension procedure} for $\Psi(\cd\,,\cd)$.

\ms

With the above preparation, we now proceed a construction of equilibrium strategies. Let $\t\in[0,T)$ be fixed and let $\Pi\equiv\{t_k~|~0\les k\les N\}$ be a partition of $[\t,T]$ with
\bel{Pi}\t=t_0<t_1<\cds<t_{N-1}<t_N=T,\ee
whose {\it mesh size} $\|\Pi\|$ is defined by
$$\|\Pi\|\deq\ds\max_{0\les i\les N-1}|t_{i+1}-t_i|.$$
For the given partition $\Pi$, denote
\bel{def-pi}\left\{\ba{ll}
\ds\pi^\Pi(t)=\sum_{k=0}^{N-2}t_k{\bf1}_{[t_k,t_{k+1})}+t_{N-1}{\bf1}_{[t_{N-1},t_N]},\\
\ns\ds\wt\pi^\Pi(t)=t_1{\bf1}_{[t_{0},t_1]}+\sum_{k=1}^{N-1}t_{k+1} {\bf1}_{(t_k,t_{k+1}]}.\ea\right.\ee
Following the idea of \cite{Yong 2012}, we now inductively construct a feedback strategy $\Psi^\Pi:[\t,T]\times\dbR^n\to U$ associated with $\Pi$ by means of optimal controls. For any $t\in[t_{N-1},T]$, we first consider the following controlled SDE:
\bel{state(N)}X^N(s)=\xi+\int_{t}^sb(r,X^N(r),u^N(r))dr+\int_{t}^s\si(r,X^N(r),u^N(r))dW(r),\q s\in[t,T],\ee
with the recursive cost functional
\bel{J^N}J^N(t,\xi;u^N(\cd))=Y^N(t),\ee
where $(Y^N(\cd),Z^N(\cd))$ is the adapted solution to the following BSDE:
\bel{BSDE(N)}\ba{ll}
\ns\ds Y^N(s)=h(t_{N-1},X^N(T))+\int_s^Tg\big(t_{N-1},r,X^N(r),u^N(r),Y^N(r),Z^N(r)\big)dr\\
\ns\ds\qq\qq\qq\qq\qq\qq-\int_s^TZ^N(r)dW(r),\qq s\in[t,T].\ea\ee
The optimal control problem associated with the above state equation \rf{state(N)} and recursive cost functional \rf{J^N}--\rf{BSDE(N)} is time-consistent. By dynamic programming approach, under proper conditions, the value function, denoted by, $V^N(\cd\,,\cd)$ is the classical solution to the following HJB equation:
\bel{HJB(N)}\left\{\2n\ba{ll}
\ns\ds V^N_s(s,x)+\inf_{u\in U}\BH\big(t_{N-1},s,x,u,V^N(s,x),V^N_x(s,x),V^N_{xx}(s,x)\big)=0,\q(s,x)\in[t_{N-1},T],\\
\ns\ds V^N(T,x)=h(t_{N-1},x),\qq x\in\dbR^n,\ea\right.\ee
with the Hamiltonian $\BH$ given by \rf{Hamilton}. Then define feedback strategy
\bel{Psi(N)}\Psi^N(s,x)=\psi\big(t_{N-1},s,x,V^N(s,x),V^N_x(s,x),V^N_{xx}(s,x)\big),
\qq(s,x)\in[t_{N-1},T]\times\dbR^n,\ee
whose outcome
$$u^N(s)=\Psi^N(s,X^N(s)),\qq s\in[t_{N-1},T]$$
is an optimal control of the corresponding optimal control problem.

\ms

Next, by the above strategy extension procedure, we obtain an extension $\Psi^{N-1}:[t_{N-2},T]\times\dbR^n\to U$ of $\Psi^N(\cd\,,\cd)$ by the following steps:

\ms

\it Step 1. \rm Solve the following representation PDE parameterized by $t\in[t_{N-1},T]$:
\bel{PDE-FBSDE(N-1)}\left\{\2n\ba{ll}
\ds\Th^{N-1}_s(t,s,x)+\BH\big(t,s,x,\Psi^{N}(s,x),\Th^{N-1}(s,s,x),\Th^{N-1}_x
(t,s,x),\Th^{N-1}_{xx}(t,s,x)\big)=0,\\
\ns\ds\qq\qq\qq\qq\qq\qq\qq\qq\qq\qq\qq(t,s,x)\in\D[t_{N-1},T]\times\dbR^n,\\
\ns\ds\Th^{N-1}(t,T,x)=h(t,x),\q(t,x)\in[t_{N-1},T]\times\dbR^n,\ea\right.\ee
and then solve the following PDE:
\bel{PDE-FBSDE(N-1)*}\left\{\2n\ba{ll}
\ds\wt\Th^{N-1}_s(s,x)+\BH\big(t_{N-2},s,x,\Psi^N(s,x),\Th^{N-1}(s,s,x),\wt\Th^{N-1}_x
(s,x),\wt\Th^{N-1}_{xx}(s,x)\big)=0,\\
\ns\ds\qq\qq\qq\qq\qq\qq\qq\qq\qq\qq\qq\qq(s,x)\in[t_{N-1},T]\times\dbR^n,\\
\ns\ds\wt\Th^{N-1}(T,x)=h(t_{N-2},x),\qq x\in\dbR^n.\ea\right.\ee

\it Step 2. \rm Solve the following HJB equation:
\bel{HJB(N-1)}\left\{\2n\ba{ll}
\ds V^{N-1}_s(s,x)+\inf_{u\in U}\BH\big(t_{N-2},s,x,u,V^{N-1}(s,x),V^{N-1}_x(s,x),V^{N-1}_{xx}(s,x)\big)=0,\\
\ns\ds\qq\qq\qq\qq\qq\qq\qq\qq\qq\qq(s,x)\in[t_{N-2},t_{N-1}]\times\dbR^n,\\
\ns\ds V^{N-1}(t_{N-1},x)=\wt\Th^{N-1}(t_{N-1},x),\qq x\in\dbR^n,\ea\right.\ee
with the Hamiltonian $\BH$ given by \rf{Hamilton}, assuming that the classical solution $V^{N-1}(\cd\,,\cd)$ exists.

\ms

\it Step 3. \rm Define
\bel{Psi(N-1)}\Psi^{N-1}(s,x)\1n=\1n\left\{\2n\ba{ll}
\ds\Psi^N(s,x),\qq\qq\qq\qq\qq\qq\qq\qq\qq(s,x)\in[t_{N-1},T]\times\dbR^n,\\
\ns\ds\psi\big(t_{N-2},s,x,V^{N-1}(s,x),V^{N-1}_x(s,x),
V^{N-1}_{xx}(s,x)\big),\q(s,x)\1n\in\1n[t_{N-2},t_{N-1})\1n\times\1n\dbR^n.\ea\right.\ee
By verification theorem (\autoref{veri-thm}), we know that the outcome
$$u^{N-1}(s)=\Psi^{N-1}(s,X^{\Psi^{N-1}}(s)),\qq s\in[t_{N-2},t_{N-1}]$$
of the feedback strategy $\Psi^{N-1}(\cd\,,\cd)$ is the optimal control for the corresponding sophisticated optimal control problem on $[t_{N-2},t_{N-1}]$.

\ms

Now, suppose $\Psi^{k+1}(\cd\,,\cd)$ has been constructed on $[t_k,T]\times\dbR^n$, (for some $k=1,2,\cds,N-1$). We apply the above strategy extension procedure to obtain an extension $\Psi^k:[t_{k-1},T]\times\dbR^n\to U$ of $\Psi^{k+1}(\cd\,,\cd)$ by the following steps:

\ms

\it Step 1. \rm Solve the following representation PDE parameterized by $t\in[t_k,T]$:
\bel{PDE-FBSDE(k)}\left\{\2n\ba{ll}
\ds\Th^{k}_s(t,s,x)+\BH\big(t,s,x,\Psi^{k+1}(s,x),\Th^{k}(s,s,x),\Th^{k}_x
(t,s,x),\Th^{k}_{xx}(t,s,x)\big)=0,\\
\ns\ds\qq\qq\qq\qq\qq\qq\qq\qq\qq\qq(t,s,x)\in\D[t_k,T]\times\dbR^n,\\
\ns\ds\Th^{k}(t,T,x)=h(t,x),\q(t,x)\in[t_k,T]\times\dbR^n,\ea\right.\ee
and then solve the following PDE:
\bel{PDE-FBSDE(k)*}\left\{\2n\ba{ll}
\ds\wt\Th^{k}_s(s,x)\1n+\1n\BH\big(t_{k-1},s,x,\Psi^{k+1}(s,x),\Th^{k}(s,s,x),\wt\Th^{k}_x
(s,x),\wt\Th^{k}_{xx}(s,x)\big)\1n=\1n0,\\
\ns\ds\qq\qq\qq\qq\qq\qq\qq\qq\qq\qq\qq(s,x)\in[t_k,T]\times\dbR^n,\\
\ns\ds\wt\Th^{k}(T,x)=h(t_{k-1},x),\qq x\in\dbR^n.\ea\right.\ee

\it Step 2. \rm Solve the following HJB equation:
\bel{HJB(k)}\left\{\2n\ba{ll}
\ds V^k_s(s,x)+\inf_{u\in U}\BH\big(t_{k-1},s,x,u,V^k(s,x),V^k_x(s,x),V^k_{xx}(s,x)\big)=0,\q(s,x)\in
[t_{k-1},t_k]\times\dbR^n,\\
\ns\ds V^k(t_k,x)=\wt\Th^{k}(t_k,x),\qq x\in\dbR^n,\ea\right.\ee
with the Hamiltonian $\BH$ given by \rf{Hamilton}, again, assuming the classical solution $V^k(\cd\,,\cd)$ exists.

\ms

\it Step 3. \rm Define
\bel{Psi(k)}\Psi^k(s,x)=\left\{\2n\ba{ll}
\ds\Psi^{k+1}(s,x),\qq\qq\qq\qq\qq\qq\qq\qq(s,x)\in[t_k,T]\times\dbR^n,\\
\ns\ds\psi\big(t_{k-1},s,x,V^k(s,x),V^k_x(s,x),V^k_{xx}(s,x)\big),
\qq\q(s,x)\in[t_{k-1},t_k)\times\dbR^n.\ea\right.\ee
The same as above, the outcome
$$u^k(s)=\Psi^k(s,X^{\Psi^k}(s)),\qq s\in[t_{k-1},t_k],$$
of the feedback strategy $\Psi^k(\cd\,,\cd)$ is an optimal control of the corresponding sophisticated optimal control on $[t_{k-1},t_k]$.

\ms

This completes the induction.

\ms

It is seen that $\Psi^0:[\t,T]\times\dbR^n\to U$ is a feedback
strategy whose outcome $u^k(\cd)$ on $[t_{k-1},t_k]$ is optimal for the corresponding sophisticated problem on $[t_{k-1},t_k]$. We call the above constructed $\Psi^0(\cd,\cd)$ (which is determined by partition $\Pi$) an {\it equilibrium strategy} associated with $\Pi$.
Our next goal is to obtain the limit as $\|\Pi\|\to0$.

\ms



To get the right ansatz, let us make an observation. Once $\Psi^k(s,x)$ is defined for $(s,x)\in[t_{k-1},T]\times\dbR^n$, we may extend
$\Th^{k}(t,s,x)$ and
$\wt\Th^{k}(s,x)$ as follows:
\bel{PDE-FBSDE(k)-}\left\{\2n\ba{ll}
\ds\Th^{k}_s(t,s,x)+\BH\big(t,s,x,\Psi^k(s,x),\Th^{k}(s,s,x),\Th^{k}_x
(t,s,x),\Th^{k}_{xx}(t,s,x)\big)=0,\\
\ns\ds\qq\qq\qq\qq\qq\qq\qq\qq\qq\qq(t,s,x)\in\D[t_{k-1},T]\times\dbR^n,\\
\ns\ds\Th^{k}(t,T,x)=h(t,x),\q(t,x)\in[t_{k-1},T]\times\dbR^n,\ea\right.\ee
and
\bel{PDE-FBSDE(k)*-}\left\{\2n\ba{ll}
\ds\wt\Th^{k}_s(s,x)\1n+\1n\BH\big(t_{k-1},s,x,\Psi^k(s,x),\Th^{k}(s,s,x),\wt\Th^{k}_x
(s,x),\wt\Th^{k}_{xx}(s,x)\big)\1n=\1n0,\q(s,x)\in[t_{k},T]\times\dbR^n,\\ [1mm]
\ds\wt\Th^{k}_s(s,x)\1n+\1n\BH\big(t_{k-1},s,x,\Psi^k(s,x),\wt\Th^{k}(s,x),\wt\Th^{k}_x
(s,x),\wt\Th^{k}_{xx}(s,x)\big)\1n=\1n0,\q(s,x)\in[t_{k-1},t_k)\times\dbR^n,\\ [1mm]
\ns\ds\wt\Th^{k}(T,x)=h(t_{k-1},x),\qq x\in\dbR^n.\ea\right.\ee
%
Then
$$V^k(s,x)=\wt\Th^{k}(s,x),\qq s\in[t_{k-1},t_k),\q x\in\dbR^n.$$
Now, for any given partition $\Pi$ of $[\t,T]$, we define
\bel{Th-Pi-tau-h-g}
\left\{\begin{aligned}
&\Th^\Pi(t,s,x)=\Th^{k}(t,s,x),\q(t,s,x)\in\D[t_{k},T]\times\dbR^n,\q k=0,1,...,N-1,\\
&\wt\Th^\Pi(t,s,x)=\sum_{k=1}^{N} \wt\Th^{k}(s,x){\bf1}_{[t_{k-1},t_k)}(t),\q (t,s,x)\in\D[\t,T]\times\dbR^n,\\
&h^\Pi(t,x)=\sum_{k=1}^N h(t_{k-1},x){\bf1}_{[t_{k-1},t_k)}(t),\q (t,x)\in[\t,T]\times\dbR^n,\\
&g^\Pi(t,s,x,u,y,z)=\sum_{k=1}^N g(t_{k-1},s,x,u,y,z){\bf1}_{[t_{k-1},t_k)}(t),\\
&\qq\qq\qq\qq (t,s,x,u,y,z)\in\D[\t,T]\times\dbR^n\times U\times\dbR\times\dbR^{1\times d}.
\end{aligned}\right.
\ee
Note that
$$
\Th^{j}(t,s,x)=\Th^{k}(t,s,x),\q(t,s,x)\in\D[t_{k},T]\times\dbR^n,\q 1\les j\les k\les N.$$
Thus $\Th^\Pi(\cd,\cd,\cd)$ is well-defined.
Then $\Th^\Pi(\cd,\cd,\cd)$ and $\wt\Th^\Pi(\cd,\cd,\cd)$ satisfy the following PDEs:
\bel{Th-Pi}\left\{\2n\ba{ll}
\ns\ds\Th^\Pi_s(t,s,x)+\BH\big(t,s,x,\Psi^\Pi(s,x),\Th^\Pi(s,s,x),\Th^\Pi_x(t,s,x),
\Th^\Pi_{xx}(t,s,x)\big)=0,\\
\ns\ds\qq\qq\qq\qq\qq\qq\qq\qq\qq\qq(t,s,x)\in\D[\t,T]\times\dbR^n,\\
\ns\ds\Th^\Pi(t,T,x)=h(t,x),\q (t,x)\in[\t,T]\times\dbR^n,\ea\right.\ee
and
\bel{Th-Pi-tau}
\left\{\begin{aligned}
& \wt\Th^{\Pi}_s(t,s,x)+\BH^\Pi\big(t,s,x,\Psi^\Pi(s,x),\Th^\Pi(s,s,x),\wt\Th^\Pi_x(t,s,x),\wt\Th^\Pi_{xx}(t,s,x)\big)=0,\\
&\qq\qq\qq\qq\qq\qq\qq\qq\qq t\in [\t,T],\,s\in[\ti\pi^\Pi(t),T],\,x\in\dbR^n,\\
& \wt\Th^{\Pi}_s(t,s,x)+\BH^\Pi\big(t,s,x,\Psi^\Pi(s,x), \wt\Th^\Pi(t,s,x),\wt\Th^\Pi_x(t,s,x),\wt\Th^\Pi_{xx}(t,s,x)\big)=0,\\
&\qq\qq\qq\qq\qq\qq\qq\qq\qq  t\in [\t,T],\,s\in[t,\ti\pi^\Pi(t)]\,,x\in\dbR^n,\\
& \wt\Th^{\Pi}(t,T,x)=h^\Pi(t,x),\q (t,x)\in[\t,T]\times\dbR^n,
\end{aligned}\right.\ee
where
\bel{H-Pi}\ba{ll}
\ns\ds\BH^\Pi(t,s,x,u,\th,p,P)={1\over 2}\tr[P\si(s,x,u)\si(s,x,u)^\top]+p b(s,x,u)+g^\Pi(t,s,x,u,\th,p\si(s,x,u)),\\
\ns\ds\qq\qq\qq\qq\qq\qq(t,s,x,u,\th,p,P)\in \D[\t,T]\times\dbR^n\times U\times\dbR\times\dbR^{1\times n}\times\dbS^n,\ea\ee
and
\bel{Psi-Pi}\Psi^\Pi(s,x)=\psi(\pi^\Pi(s),s,x,\wt\Th^\Pi(s,s,x),\wt\Th^\Pi_x(s,s,x),
\wt\Th^\Pi_{xx}(s,s,x)),\q (s,x)\in[\t,T]\times\dbR^n.\ee
Note that
\bel{h-h}|h^\Pi(t,x)-h(t,x)|=\sum_{k=1}^N|h(t_{k-1},x)-h(t,x)|{\bf1}_{[t_{k-1},t_k)}(t)
\les L\|\Pi\|,\ee
and likewise,
\bel{H-H}\ba{ll}
\ns\ds|\BH^\Pi(t,s,x,u,\th,p,P)-\BH(t,s,x,u,\th,p,P)|\\
\ns\ds=|g^\Pi(t,s,x,u,\th,p\si(s,x,u))-g(t,s,x,u,\th,p\si(s,x,u))|\les L\|\Pi\|.\ea\ee
Therefore, at the limit, one should have
\bel{PDE-FBSDE(limit)}\left\{\2n\ba{ll}
\ds\Th_s(t,s,x)\1n+\1n\BH\big(t,s,x,\Psi(s,x),\Th(s,s,x),\Th_x(t,s,x),\Th_{xx}(t,s,x)\big)=0,
\q(t,s,x)\1n\in\1n\D[\t,T]\1n\times\1n\dbR^n,\\
\ns\ds\Th(t,T,x)=h(t,x),\q(t,x)\in[\t,T]\times\dbR^n,\ea\right.\ee
\bel{tiPDE-FBSDE(limit)}\left\{\2n\ba{ll}
\ds\wt\Th_s(t,s,x)\1n+\1n\BH\big(t,s,x,\Psi(s,x),\Th(s,s,x),\wt\Th_x(t,s,x),\wt\Th_{xx}(t,s,x)\big)=0,
\q(t,s,x)\1n\in\1n\D[\t,T]\1n\times\1n\dbR^n,\\
\ns\ds\Th(t,T,x)=h(t,x),\q(t,x)\in[\t,T]\times\dbR^n,\ea\right.\ee
\bel{V(limit)}V(s,x)=\wt\Th(s,s,x),\qq(s,x)\in[\t,T]\times\dbR^n,\ee
and
\bel{Psi(limit)}\Psi(s,x)=\psi\big(s,s,x,\wt\Th(s,s,x),\wt\Th_x(s,s,x),\wt\Th_{xx}(s,s,x)\big),\qq(s,x)\in
[\t,T]\times\dbR^n.\ee
Let \rf{PDE-FBSDE(limit)} admit a unique classical solution. Then
\bel{ti-Th-Th}
\wt\Th(t,s,x)=\Th(t,s,x),\q (t,s,x)\in\D[\t,T]\times\dbR^n,
\ee
and thus
\begin{align}
\label{V(limit)-Th-Psi1}&V(s,x)=\wt\Th(s,s,x)=\Th(s,s,x),\qq(s,x)\in[\t,T]\times\dbR^n,\\
\label{V(limit)-Th-Psi}&\Psi(s,x)=\psi\big(s,s,x,\Th(s,s,x),\Th_x(s,s,x),\Th_{xx}(s,s,x)\big)
\qq(s,x)\in[\t,T]\times\dbR^n.
\end{align}

\ms

We call \rf{PDE-FBSDE(limit)} the {\it equilibrium HJB equation} of Problem (N), and $V(\cd,\cd)$ the {\it equilibrium value function} of Problem (N). The map $\Psi(\cd,\cd)$ defined by \rf{V(limit)-Th-Psi} is a {\it feedback strategy} of Problem (N) provided that \rf{PDE-FBSDE(limit)} has a solution with good regularities. We will show that the feedback strategy $\Psi(\cd,\cd)$ is an {\it equilibrium strategy} of Problem (N) in the next section. Note that the equilibrium HJB equation \rf{PDE-FBSDE(limit)} can actually be written as follows:
\bel{Th-differential-form-rewrite}\left\{\2n\ba{ll}
\ds\Th_s(t,s,x)+{1\over2}\tr\[\Th_{xx}(t,s,x)a\big(s,x,\psi(s,s,x,\Th(s,s,x),\Th_x(s,s,x),
\Th_{xx}(s,s,x))\big)\]\\
\ns\ds\q+ \Th_x(t,s,x)b\big(s,x,\psi(s,s,x,\Th(s,s,x),\Th_x(s,s,x),\Th_{xx}(s,s,x))\big)\\
\ns\ds\q+g\big(t,s,x,\psi(s,s,x,\Th(s,s,x),\Th_x(s,s,x),\Th_{xx}(s,s,x)),\Th(s,s,x),\\
\ns\ds\qq\qq\Th_x(t,s,x)\si(s,x,\psi(s,s,x,\Th(s,s,x),\Th_x(s,s,x),\Th_{xx}(s,s,x)))\big)=0,\\
\ns\ds\qq\qq\qq\qq\qq\qq\qq\qq\qq\qq\qq\q(t,s,x)\in\D[\t,T]\times\dbR^n,\\
\ns\ds\Th(t,T,x)=h(t,x),\q (t,x)\in[\t,T]\times\dbR^n,\ea\right.\ee
where
\bel{def-a}
a(s,x,u)\deq\si(s,x,u)\si(s,x,u)^\top,\q (s,x,u)\in[\t,T]\times\dbR^n\times U.
\ee
%
Taking the feedback strategy $\Psi(\cd\,,\cd)$ defined by \rf{V(limit)-Th-Psi}, the corresponding closed-loop system is
\bel{equilibrium-system}\left\{\begin{aligned}
\bar X(s) &=\xi+\int_\t^sb\big(r,\bar X(r),\Psi(r,\bar X(r))\big)dr+\int_\t^s\si\big(r,\bar X(r),\Psi(r,\bar X(r))\big)dW(r),\qq s\in[\t,T],\\
\bar Y(s) &=h(s,\bar X(T))+\int_s^T g\big(s,r,\bar X(r),\Psi(r,\bar X(r)),\bar Y(r),\bar Z(s,r)\big)dr\\
  &\qq\qq\qq\qq\qq -\int_s^T \bar Z(s,r)dW(r),\qq s\in[\t,T].
\end{aligned}\right.\ee
The equilibrium value function is
\bel{bar Y}V(s,\bar X(s))=\Th(s,s,\bar X(s))=\bar Y(s),\q s\in[\t,T].\ee
Also,
\bel{bar Z}\bar Z(t,s)=\Th_x(t,s,\bar X(s))\si\big(s,\bar X(s),\Psi(s,\bar X(s))\big),\qq (t,s)\in\D[\t,T].\ee
We see that \rf{bar Y}--\rf{bar Z} exhibits an interesting relationship between equilibrium HJB equation \rf{Th-differential-form-rewrite} and coupled FSDE and BSVIE \rf{equilibrium-system}. We will explore more about this in  \autoref{BSVIEs}.

\section{Verification Theorem and Well-Posedness of Equilibrium HJB Equation}\label{Verification Theorem}
For any partition $\Pi$, we have constructed an {\it approximate equilibrium strategy} $\Psi^\Pi(\cd,\cd)$ of Problem (N). Taking the limit $\ds\lim_{\|\Pi\|\to 0}\Psi^\Pi(\cd,\cd)$, we have formally obtained the feedback strategy $\Psi(\cd,\cd)$.
In this section, we would like to show that $\Psi(\cd,\cd)$   is  an equilibrium
strategy of Problem (N) in the sense of \autoref{equilibrium strategy}.
Such a result can be viewed as a verification theorem for our constructed strategy $\Psi(\cd,\cd)$.

\ms

In order to show the local optimality of the feedback strategy $\Psi(\cd,\cd)$,
we assume that the equilibrium HJB equation \rf{Th-differential-form-rewrite} admits a unique smooth solution. We also assume that all the involved functions are bounded and differentiable with bounded derivatives.
For any fixed $ t \in[0,T)$  and $\e>0$ small with $t+\e\les T$, we consider Problem (C$^\Psi[t,t+\e]$).
Then by \autoref{time-consistent}, the state equation and the cost functional of Problem (C$^\Psi[t,t+\e]$) can be given by
\bel{state-t-t+e}
X(s)=\xi+\int_t^sb(r,X(r),u(r))dr+\int_t^s\si(r,X(r),u(r))dW(r),\q s\in[t,t+\e],
\ee
and
\bel{J-t}
J^t(t,\xi;u(\cd))=\wt Y(t),
\ee
with $(\wt Y(\cd),\wt Z(\cd))$ being the adapted solution of the following BSDE:
\bel{BSDE-wtY-t+e}\ba{ll}
\ns\ds\wt Y(s)=\Th(t,t+\e,X(t+\e))+\int_s^{t+\e}g\big(t,r,X(r),u(r),\wt Y(r),
\wt Z(r)\big)dr\\
\ns\ds\qq\qq\qq\qq\qq\qq\qq-\int_s^{t+\e}\wt Z(r)dW(r),\q s\in[t,t+\e].
\ea\ee
Note that Problem (C$^\Psi[t,t+\e]$) is a classical recursive stochastic optimal control problem and thus is time consistent.
Let $\Th^\e(t,\cd,\cd)$ be the unique classical solution of the following HJB equation:
\bel{PDE-Th-e}\left\{\2n\ba{ll}
\ds\Th^\e_s(t,s,x)\1n+\1n\BH\big(t,s,x,\Psi^\e(s,x),\Th^\e(t,s,x),\Th^\e_x(t,s,x),
\Th^\e_{xx}(t,s,x)\big)=0,
\q(s,x)\1n\in\1n[t,t+\e]\1n\times\1n\dbR^n,\\
\ns\ds\Th^\e(t,t+\e,x)=\Th(t,t+\e,x),\q x\in\dbR^n,\ea\right.\ee
where
\bel{Psi-e}
\Psi^\e(s,x)=\psi\big(t,s,x,\Th^\e(t,s,x),\Th^\e_x(t,s,x),\Th^\e_{xx}(t,s,x)\big),\q(s,x)\1n\in\1n[t,t+\e]\1n\times\1n\dbR^n.
\ee
By \autoref{veri-thm}, the outcome $u^\e(\cd)=\Psi^\e(\cd,X^\e(\cd))$ of strategy $\Psi^\e(\cd,\cd)$ is an optimal control of Problem (C$^\Psi[t,t+\e]$);
that is
\bel{J-t-Psit}
J^t(t,\xi;\Psi^\e(\cd,\cd))=\inf_{u(\cd)\in\sU[t,t+\e]}J^t(t,\xi;u(\cd)).
\ee
Note that $\Th(\cd,\cd,\cd)$ satisfies the following equation on $\D[t,t+\e]$:
\bel{PDE-Th-t-t+e}\3n\left\{\2n\ba{ll}
\ds\Th_s( \t,s,x)\1n+\1n\BH\big(\t,s,x,\Psi(s,x),\Th(s,s,x),\Th_x(\t,s,x),\Th_{xx}(\t,s,x)
\big)\1n=\1n0,~(\t,s,x)\1n\in\1n\D[t,t\1n+\1n\e]\1n\times\1n\dbR^n,\\
\ns\ds\Th(\t,t+\e,x)=\Th(\t,t+\e,x),\q(\t,x)\in[t,t+\e]\times\dbR^n,\ea\right.\ee
where
\bel{Psi-t-t+e}
\Psi(s,x)=\psi\big(s,s,x,\Th(s,s,x),\Th_x(s,s,x),\Th_{xx}(s,s,x)\big),\q(s,x)\1n\in\1n
[t,t+\e]\1n\times\1n\dbR^n.\ee
We observe that the $\BH$-term of \rf{PDE-Th-t-t+e} depends on the diagonal values $\Th(s,s,x),\Th_x(s,s,x),\Th_{xx}(s,s,x)$ of the unknown variables and the $\BH$-term of \rf{PDE-Th-e} depends on
$\Th^\e(t,s,x),\Th^\e_x(t,s,x),\Th^\e_{xx}(t,s,x)$. Thus \rf{PDE-Th-t-t+e} is a non-local PDE and  \rf{PDE-Th-e} is a classical PDE with the parameter $t$.
Then in general, we do not have
\bel{Th-The-eq}
\Th^\e(t,s,x)=\Th(t,s,x),\q(s,x)\1n\in\1n[t,t+\e]\1n\times\1n\dbR^n.
\ee
Therefore, there are some gaps in the proofs of the verification theorem in \cite{Wei-Yong-Yu 2017,Mei-Yong 2019}.
In fact, if  \rf{Th-The-eq} holds,  by \rf{PDE-Th-e}--\rf{PDE-Th-t-t+e}, we get
\begin{align*}
&\BH\big(t,s,x,\Psi^\e(s,x),\Th^\e(t,s,x),\Th^\e_x(t,s,x),\Th^\e_{xx}(t,s,x)\big)\\
&\q=\BH\big(t,s,x,\Psi(s,x),\Th(s,s,x),\Th_x(t,s,x),\Th_{xx}(t,s,x)\big),\q(s,x)\in[t,t+\e]\times\dbR^n.
\end{align*}
Then the following should hold true:
$$
\Th^\e(t,s,x)=\Th(s,s,x),\q(s,x)\1n\in\1n[t,t+\e]\1n\times\1n\dbR^n.
$$
Combining the above with \rf{Th-The-eq} yields that
\bel{Th-Th-eq}
\Th(t,s,x)=\Th(s,s,x),\q(s,x)\1n\in\1n[t,t+\e]\1n\times\1n\dbR^n.
\ee
Since the $\BH$-term of \rf{PDE-Th-t-t+e} depends on $\t;\,t\les\t\les t+\e$, the above equality usually fails.
Therefore, we do not have \rf{Th-The-eq} in general.

\ms
By \autoref{lemma-BSVIE-approximate-estimate}, we know that there exists a constant $K>0$, independent of $(t,\xi,u(\cd))$ such that
\bel{J-t-J}
\big|J^t(t,\xi;u(\cd))-J(t,\xi;u\oplus\Psi|_{[t+\e,T]})\big|\les K\e^{2}.
\ee
Combining the above with \rf{J-t-Psit}, we have
\bel{J-Psit}
J(t,\xi;\Psi^\e\oplus\Psi|_{[t+\e,T]})\les J(t,\xi;u\oplus\Psi|_{[t+\e,T]})
+o(\e),\qq\forall u(\cd)\in\sU[t,t+\e],\ee
where $o(\e)$ is uniform in $u(\cd)\in\sU[t,t+\e]$.
Thus $(\Psi^\e\oplus\Psi|_{[t+\e,T]})(\cd,\cd)$ satisfies the local near-optimality \rf{near-optim}.
Next, we would like to show that $\Psi(\cd,\cd)$ satisfies the local optimality condition \rf{near-optim*} under the following assumption:

\ms

{\bf(H4)} There exists a nondecreasing continuous function $\wt\rho:[0,\i)\to[0,\i)$ with $\wt\rho(0)=0$ such that
$$\ba{ll}
\ns\ds|\Th^\e(t,s,x)-\Th(t,s,x)|+|\Th_x^\e(t,s,x)-\Th_x(t,s,x)|+|\Th_{xx}^\e(t,s,x)-\Th_{xx}(t,s,x)|\les \wt\rho(\e)(1+|x|),\\
\ns\ds\qq\qq\qq\qq\qq\qq\qq\qq\qq\qq\forall(s,x)\in[t,t+\e]\times\dbR^n,\ea$$
where $\Th(\cd,\cd,\cd)$ and $\Th^\e(\cd,\cd,\cd)$ are the classical solutions of  the PDE \rf{PDE-Th-t-t+e} and \rf{PDE-Th-e}, respectively.

\ms

Under (H4), by the definitions   of $\Psi^\e(\cd,\cd)$  and $\Psi(\cd,\cd)$, we have
\bel{Psi-e-Psi}
\begin{aligned}
&\q|\Psi^\e(s,x)-\Psi(s,x)|\\
&=\big|\psi\big(t,s,x,\Th^\e(t,s,x),\Th^\e_x(t,s,x),\Th^\e_{xx}(t,s,x)\big)-\psi\big(s,s,x,\Th(s,s,x),\Th_x(s,s,x),\Th_{xx}(s,s,x)\big)\big|\\
&\les \big|\psi\big(t,s,x,\Th^\e(t,s,x),\Th^\e_x(t,s,x),\Th^\e_{xx}(t,s,x)\big)-\psi\big(t,s,x,\Th(t,s,x),\Th_x(t,s,x),\Th_{xx}(t,s,x)\big)\big|\\
&\q+\big|\psi\big(t,s,x,\Th(t,s,x),\Th_x(t,s,x),\Th_{xx}(t,s,x)\big)-\psi\big(s,s,x,\Th(s,s,x),\Th_x(s,s,x),\Th_{xx}(s,s,x)\big)\big|\\
&\les K\wt\rho(\e)(1+|x|)+K|s-t|\les K[\wt\rho(\e)+\e](1+|x|)\deq\rho(\e)(1+|x|),\q (s,x)\in[t,t+\e]\times\dbR^n.
\end{aligned}
\ee
%
Let $(X(\cd),\wt Y(\cd),\wt Z(\cd))\equiv(X^\Psi(\cd),\wt Y^\Psi(\cd),\wt Z^\Psi(\cd))$ and $(X^\e(\cd),\wt Y^\e(\cd),\wt Z^\e(\cd))\equiv (X^{\Psi^\e}(\cd),\wt Y^{\Psi^\e}(\cd),\wt Z^{\Psi^\e}(\cd))$ be the unique solutions to the controlled SDE \rf{state-t-t+e}
and BSDE \rf{BSDE-wtY-t+e} corresponding to the feedback strategies $\Psi(\cd,\cd)$ and $\Psi^\e(\cd,\cd)$, respectively.
By \autoref{lmm:well-posedness-SDE} and \rf{Psi-e-Psi}, we have
\bel{X-e-est1}
\begin{aligned}
\dbE_t\(\sup_{t\les s\les t+\e}|X^\e(s)|^2\)&\les K\dbE_t\Big\{1+|\xi|^2+ \int_t^{t+\e}\big|\Psi^\e(r,X^\e(r))\big|^2dr\Big\}\\
&\les K\dbE_t\Big\{1+|\xi|^2+ \int_t^{t+\e}\big[|\Psi(r,X^\e(r))|^2+\rho(\e)^2(1+|X^\e(r)|^2)\big]dr\Big\}\\
&\les K\dbE_t\Big\{1+|\xi|^2+ \int_t^{t+\e}\big[1+|X^\e(r)|^2+\rho(T-t)^2(1+|X^\e(r)|^2)\big]dr\Big\}\\
&\les K\dbE_t\Big\{1+|\xi|^2+ \int_t^{t+\e}|X^\e(r)|^2dr\Big\}.
\end{aligned}
\ee
Then by Gr\"{o}nwall's inequality, there exists a constant $K>0$, independent of $(t,\e)$ such that
\bel{X-e-est}
\dbE_t\(\sup_{t\les s\les t+\e}|X^\e(s)|^2\)\les  K\dbE_t[1+|\xi|^2]=K(1+|\xi|^2).
\ee
Combining the above with the standard estimate of SDEs, we have
\bel{X-e-X-est}
\begin{aligned}
&\dbE_t\(\sup_{t\les s\les t+\e}|X^\e(s)-X(s)|^2\)\\
&\q\les K\dbE_t\Big\{ \int_t^{t+\e}\big|b\big(r,X(r),\Psi^\e(r,X^\e(r))\big)-b\big(r,X(r),\Psi(r,X^\e(r))\big)\big|^2dr\\
&\qq\qq+\int_t^{t+\e}\big|\si\big(r,X(r),\Psi^\e(r,X^\e(r))\big)-\si\big(r,X(r),\Psi(r,X^\e(r))\big)\big|^2dr\Big\}\\
&\q\les K\dbE_t\Big\{\int_t^{t+\e}\rho(\e)^2(1+|X^\e(r)|^2)dr\Big\}\les  K\e\rho(\e)^2(1+|\xi|^2).
\end{aligned}
\ee
In particular,
\bel{X-e-X-t-est}
\dbE_t|X^\e(t+\e)-X(t+\e)|^2\les  K\e\rho(\e)^2(1+|\xi|^2).
\ee
Let us recall \rf{BSDE-wtY-t+e}. Then  by the standard estimate of BSDEs  and \rf{X-e-X-est}--\rf{X-e-X-t-est}, we have
\bel{Y-e-Y-est}
\begin{aligned}
&\dbE_t\[\sup_{t\les s\les t+\e}|\wt Y^\e(s)-\wt Y(s)|^2\]+\dbE_t\[\int_t^{t+\e}|\wt Z^\e(r)-\wt Z(r)|^2dr\]\\
&\q\les K\dbE_t\[|\Th(t,t+\e,X^\e(t+\e))-\Th(t,t+\e,X(t+\e))|^2\]\\
&\qq+ K\dbE_t\Big[\int_t^{t+\e}\big|g\big(t,r,X^\e(r),\Psi^\e(r,X^\e(r)),\wt Y(r),\wt Z(r)\big)\\
&\qq\qq\qq\qq-g\big(t,r,X(r),\Psi(r,X(r)),\wt Y(r),\wt Z(r)\big)\big|dr\Big]^2\\
&\q\les  K\e\rho(\e)^2(1+|\xi|^2)+K\e\dbE_t\int_t^{t+\e}\Big[|X^\e(r)-X(r)|^2+|\Psi^\e(r,X^\e(r))-\Psi(r,X^\e(r))|^2\\
&\qq\qq\qq\qq\qq\qq\qq\qq\qq+|\Psi(r,X^\e(r))-\Psi(r,X(r))|^2\Big]dr\\
&\q\les  K\e\rho(\e)^2(1+|\xi|^2)+K\e\dbE_t\int_t^{t+\e}\big[|X^\e(r)-X(r)|^2+|\Psi^\e(r,X^\e(r))-\Psi(r,X^\e(r))|^2\big]dr\\
&\q\les  K\e\rho(\e)^2(1+|\xi|^2)+ K\e^2\rho(\e)^2(1+|\xi|^2)+K\e\dbE_t\int_t^{t+\e}\rho(\e)^2\big(1+|X^\e(r)|^2\big)dr\\
&\q\les K\e\rho(\e)^2(1+|\xi|^2).
\end{aligned}
\ee
%
%
Observe that
$$
\begin{aligned}
&|\wt Y^\e(t)-\wt Y(t)|^2= \Big|\dbE_t\Big\{\Th(t,t+\e,X^\e(t+\e))-\Th(t,t+\e,X(t+\e))\\
&\qq\qq\qq\qq+\int_t^{t+\e}\[g\big(t,r,X^\e(r),\Psi^\e(r,X^\e(r)),\wt Y^\e(r),\wt Z^\e(r)\big)\\
&\qq\qq\qq\qq\qq\q- g\big(t,r,X(r),\Psi(r,X(r)),\wt Y(r),\wt Z(r)\big)\]dr\Big\}\Big|^2.
\end{aligned}
$$
Then by H\"{o}lder inequality and \rf{Psi-e-Psi}--\rf{X-e-X-est}--\rf{Y-e-Y-est}, we have
\bel{Y-Y-e1}
\begin{aligned}
|\wt Y^\e(t)-\wt Y(t)|^2&\les K\big|\dbE_t\big[\Th(t,t+\e,X^\e(t+\e))-\Th(t,t+\e,X(t+\e))\big] \big|^2\\
&\q+ K\e\dbE_t\int_t^{t+\e}\big|g\big(t,r,X^\e(r),\Psi^\e(r,X^\e(r)),\wt Y^\e(r),\wt Z^\e(r)\big)\\
&\qq\qq\qq\q- g\big(t,r,X(r),\Psi(r,X(r)),\wt Y(r),\wt Z(r)\big)\big|^2dr\\
&\les K\big|\dbE_t\big[\Th(t,t+\e,X^\e(t+\e))-\Th(t,t+\e,X(t+\e))\big] \big|^2\\
&\q+ K\e\dbE_t\int_t^{t+\e}\big[|X(r)-X^\e(r)|^2+|\Psi^\e(r,X^\e(r))-\Psi(r,X^\e(r))|^2\\
&\qq\qq\qq\qq+|\wt Y^\e(r)-\wt Y(r)|^2+|\wt Z^\e(r)-\wt Z(r)|^2\big]dr\\
&\les K\big|\dbE_t\big[\Th(t,t+\e,X^\e(t+\e))-\Th(t,t+\e,X(t+\e))\big] \big|^2+K\e^2\rho(\e)^2(1+|\xi|^2).
\end{aligned}
\ee
Applying It\^{o} formula to $s\mapsto \Th(t,t+\e,X(s))$ on $[t,t+\e]$ implies that
\bel{Ito-Th-X}
\begin{aligned}
&\Th(t,t+\e,X(t+\e))\\
&\q=\Th(t,t+\e,\xi)+\int_t^{t+\e}\Big\{\Th_x(t,t+\e,X(r)) b\big(r,X(r),\Psi(r,X(r))\big)\\
&\qq\q+{1\over2}\tr\big[\Th_{xx}(t,t+\e,X(r))\si\big(r,X(r),\Psi(r,X(r))\big)\si\big(r,X(r),\Psi(r,X(r))\big)^\top\big]\Big\}dr\\
&\qq+\int_t^{t+\e}\Th_x(t,t+\e,X(r))\si\big(r,X(r),\Psi(r,X(r))\big)dW(r).
\end{aligned}
\ee
Similarly,
\bel{Ito-Th-X-e}
\begin{aligned}
&\Th(t,t+\e,X^\e(t+\e))\\
&\q=\Th(t,t+\e,\xi)+\int_t^{t+\e}\Big\{\Th_x(t,t+\e,X^\e(r)) b\big(r,X^\e(r),\Psi^\e(r,X^\e(r))\big)\\
&\qq\q+{1\over2}\tr\big[\Th_{xx}(t,t+\e,X^\e(r))\si\big(r,X^\e(r),\Psi^\e(r,X^\e(r))\big)\si\big(r,X^\e(r),\Psi^\e(r,X^\e(r))\big)^\top\big]\Big\}dr\\
&\qq+\int_t^{t+\e}\Th_x(t,t+\e,X^\e(r))\si\big(r,X^\e(r),\Psi^\e(r,X^\e(r))\big)dW(r).
\end{aligned}
\ee
Thus, we get
\begin{align*}
&\big|\dbE_t\big[\Th(t,t+\e,X^\e(t+\e))-\Th(t,t+\e,X(t+\e))\big]\big|^2\\
&=\Big|\dbE_t\int_t^{t+\e}\3n\Big\{\Th_x(t,t+\e,X(r))^\top b\big(r,X(r),\Psi(r,X(r))\big)-\Th_x(t,t+\e,X^\e(r))^\top b\big(r,X^\e(r),\Psi^\e(r,X^\e(r))\big)\\
&\qq+{1\over2}\tr\big[\Th_{xx}(t,t+\e,X(r))\si\big(r,X(r),\Psi(r,X(r))\big)\si\big(r,X(r),\Psi(r,X(r))\big)^\top\big]\\
&\qq-{1\over2}\tr\big[\Th_{xx}(t,t+\e,X^\e(r))\si\big(r,X^\e(r),\Psi^\e(r,X^\e(r))\big)\si\big(r,X^\e(r),\Psi^\e(r,X^\e(r))\big)^\top\big]\Big\}dr\Big|^2\\
&\les K\e\dbE_t\int_t^{t+\e}\3n\big|\Th_x(t,t\1n+\1n\e,X(r))^\top\1n b\big(r,X(r),\Psi(r,X(r))\big)\1n-\1n\Th_x(t,t\1n+\1n\e,X^\e(r))^\top\1n b\big(r,X^\e(r),\Psi^\e(r,X^\e(r))\big)\\
&\qq+{1\over2}\tr\big[\Th_{xx}(t,t+\e,X(r))\si\big(r,X(r),\Psi(r,X(r))\big)\si\big(r,X(r),\Psi(r,X(r))\big)^\top\big]\\
&\qq-{1\over2}\tr\big[\Th_{xx}(t,t+\e,X^\e(r))\si\big(r,X^\e(r),\Psi^\e(r,X^\e(r))\big)\si
\big(r,X^\e(r),\Psi^\e(r,X^\e(r))\big)^\top\big]\big|^2dr\\
&\les K\e\dbE_t\int_t^{t+\e} \big[|\Psi^\e(r,X^\e(r))-\Psi(r,X^\e(r))|^2+|X^\e(r)-X(r)|^2\big]dr.
\end{align*}
Then by \rf{Psi-e-Psi} and \rf{X-e-X-est}, the above implies that
\bel{Th-Th-e}
\big|\dbE_t\big[\Th(t,t+\e,X^\e(t+\e))-\Th(t,t+\e,X(t+\e))\big]\big|^2
\les K\e^2\rho(\e)^2(1+|\xi|^2).
\ee
Substituting \rf{Th-Th-e} into \rf{Y-Y-e1}, we have
$$
|J^t(t,\xi;\Psi^\e(\cd,\cd))-J^t(t,\xi;\Psi(\cd,\cd))|^2=|\wt Y^\e(t)-\wt Y(t)|^2\les K\e^2\rho(\e)^2(1+|\xi|^2).
$$
Combining the above with \rf{J-t-Psit}, we get
$$
J^t(t,\xi;\Psi(\cd,\cd))\les\inf_{u(\cd)\in\sU[t,t+\e]}J^t(t,\xi;u(\cd))+o(\e).
$$
Then by \autoref{lemma-BSVIE-approximate-estimate}, we have the following local near optimality of $\Psi(\cd,\cd)$:
\bel{J-Psi-final}
J\big(t,\xi;\Psi\big|_{[t,T]}\big)\les J\big(t,\xi;u\oplus\Psi|_{[t+\e,T]}\big)
+o(\e),\qq\forall u(\cd)\in\sU[t,t+\e].
\ee
In conclusion, we can state the following result formally.

\begin{theorem}\label{equi-stra} \sl
Feedback strategy $\Psi(\cd,\cd)$ defined by \rf{V(limit)-Th-Psi} is an  equilibrium strategy of Problem (N).
\end{theorem}

Since $\Psi(\cd\,,\cd)$ is an  equilibrium strategy of Problem (N), the corresponding closed-loop system \rf{equilibrium-system}
is called an {\it equilibrium system}.

\ms

We now look at the well-posedness of the equilibrium HJB equation. To this end, we first observe the expression
$$\Psi(s,x)=\psi\big(s,s,x,\Th(s,s,x),\Th_x(s,s,x),\Th_{xx}(s,s,x)\big),\qq(s,x)\in[0,T]
\times\dbR^n.$$
From the definition \rf{define-psi} of $\psi(\cd)$, it is clear that the dependence of $\si(\cd)$ on the control process $u(\cd)$ leads to the appearance of $\Th_{xx}(s,s,x)$ in $\Psi(s,x)$, which turns out to bring some essential difficulties in establishing the well-posedness of equilibrium HJB equation. At the moment, such a general situation is widely open and will be investigated in our future publications. In the subsequent analysis of this section, we will consider a special but still important case, in which $\si(\cd)$ is independent of the control process $u(\cd)$. More precisely, we assume that
\bel{si}\si(t,x,u)=\si(t,x),\q (t,x,u)\in[0,T]\times\dbR^n\times U.\ee
Under \rf{si}, (H3) becomes the following:

\ms

{\bf(H3)$'$} There is a continuous map $\psi:\D[0,T]\times\dbR^n\times\dbR\times\dbR^{1\times n}\to U$ such that
\bel{define-psi'}\ba{ll}
\ns\ds pb(s,x,\psi(t,s,x,\th,p))+g(t,s,x,\psi(t,s,x,\th,p),\th,p\si(s,x))\\
\ns\ds=\min_{u\in U}\big[pb(s,x,u)+g(t,s,x,u,\th,p\si(s,x))\big],\q(t,s,x,\th,p)\in \D[0,T]\times\dbR^n\times\dbR\times\dbR^{1\times n}.\ea\ee

\ms

Under (H3)$'$, the equilibrium HJB equation \rf{Th-differential-form-rewrite} (with $\t=0$) reads
\bel{Th-differential-form-si}
\left\{\begin{aligned}
& \Th_s(t,s,x)+{1\over2}\tr[\Th_{xx}(t,s,x)a(s,x)]+\Th_{x}(t,s,x) b\big(s,x,\psi(s,s,x,\Th(s,s,x),\Th_x(s,s,x))\big)\\
&\qq +g\big(t,s,x,\psi(s,s,x,\Th(s,s,x),\Th_x(s,s,x)),\Th(s,s,x),\Th_x(t,s,x) \si(s,x)\big)=0,\\
&\qq\qq\qq\qq\qq\qq\qq\qq\qq (t,s,x)\in \D[0,T]\times\dbR^n,\\
& \Th(t,T,x)=h(t,x),\q (t,x)\in[0,T]\times\dbR^n,
\end{aligned}\right.\ee
where $a(\cd)$ is defined by \rf{def-a}. In the current case, the equilibrium strategy is given by
\bel{Psi*1}\Psi(s,x)=\psi(s,s,x,\Th(s,s,x),\Th_x(s,s,x)),\qq(s,x)\in[0,T]\times\dbR^n,\ee
for which $\Th_{xx}(\cd\,,\cd\,,\cd)$ does not appear.

\ms

For the well-posedness of \rf{Th-differential-form-si}, we make the following assumption.
\begin{taggedassumption}{(H5)}\label{ass:A1PDE} \rm
The maps
$$\left\{\2n\ba{ll}
\ds(s,x,\th,p)\mapsto b(s,x,\psi(s,s,x,\th,p)),\qq(s,x)\mapsto a(s,x),\\
\ns\ds(t,s,x,\th,p,\hat p)\mapsto g(t,s,x,\psi(s,s,x,\th,p),\th,\hat p\si(s,x)),
\qq(s,x)\mapsto h(s,x)\ea\right.$$
are bounded, have all required differentiability with bounded derivatives.
Moreover, there exist two constants $\l_0,\l_1>0$ such that
$$\l_0 I\les a(t,x)\les \l_1 I,\q \forall~(t,x)\in[0,T]\times\dbR^n.$$
\end{taggedassumption}

We have the following result whose proof can be found in \cite[Theorem 6.1]{Wei-Yong-Yu 2017}.
\begin{theorem}\label{theorem-PDE} \sl Let {\rm (H5)} hold. Then PDE \rf{Th-differential-form-si} admits a unique classical solution.
\end{theorem}

In the proof of \autoref{equi-stra}, we introduce the assumption (H4) to get the  local near optimality of the equilibrium strategy $\Psi(\cd,\cd)$.
When $\si(\cd)$ is independent of the control $u(\cd)$ (see \rf{si}),
the arguments of \autoref{equi-stra} still hold true with Assumption (H4)  replaced by the following assumption:

\ms

{\bf(H4)$'$} There exists a nondecreasing continuous function $\rho:[0,\i)\to[0,\i)$ with $\rho(0)=0$ such that
$$|\Th^\e(t,s,x)-\Th(t,s,x)|+|\Th_x^\e(t,s,x)-\Th_x(t,s,x)|\les \rho(\e)(1+|x|),\qq\forall(s,x)\in[t,t+\e]\times\dbR^n.$$

We shall show that the above assumption is a consequence of (H5).
In fact, under (H5), it follows from \autoref{theorem-PDE} that $\Th(\cd,t+\e,\cd)$ is well-defined and belongs to $C^{1,2}([t,t+\e]\times\dbR^n;\dbR)$.
Thus one can regard \rf{PDE-Th-t-t+e} as a new equilibrium HJB equation  satisfying (H5) with $h(\cd,\cd)$ and $[0,T]$ replaced by $\Th(\cd,t+\e,\cd)$ and $[t,t+\e]$, respectively. Moreover, we see that PDE \rf{PDE-Th-e} is the approximate equation of \rf{PDE-Th-t-t+e} with the partition $\Pi:t=t_0<t_1=t+\e$. Then by the last inequality in the proof of \cite[Theorem 6.2]{Wei-Yong-Yu 2017}, we have
$$
|\Th^\e(t,s,x)-\Th(t,s,x)|+|\Th_x^\e(t,s,x)-\Th_x(t,s,x)|\les K\|\Pi\|=K\e,\qq\forall(s,x)\in[t,t+\e]\times\dbR^n,
$$
for some constant $K>0$, which implies that the assumption (H4)$'$  holds.
Therefore, under (H1), (H2), (H3)$'$,  and (H5), Problem (N) admits an equilibrium strategy over $[0,T]$.

\section{BSVIEs with Diagonal Values of $Z(\cd\,,\cd)$.}\label{BSVIEs}

In this section, we look at a new type BSVIE resulted from the equilibrium solution to Problem (N). Let (H5) hold with $d=n$ and \rf{si} hold, namely, $\si$ is independent of $u$. Thus, $\si(s,x)$ is invertible. Let $\Th(\cd\,,\cd\,,\cd)$ be the classical solution to the equilibrium HJB equation \rf{Th-differential-form-si}. Then $\Psi(\cd\,,\cd)$ defined by \rf{Psi*1} is an equilibrium strategy. Substituting this $\Psi(\cd\,,\cd)$ into \rf{equilibrium-system} leads to the following closed-loop system:
\bel{FSDE-BSVIE}\left\{\2n\ba{ll}
\ds\bar X(s)=\xi+\int_\t^sb\big(r,\bar X(r),\psi(r,r,\bar X(r),
\Th(r,r,\bar X(r)),\Th_x(r,r,\bar X(r)))\big)dr\\
\ns\ds\qq\qq\qq\qq+\int_\t^s\si(r,\bar X(r))dW(r),\\
\ns\ds\bar Y(s)\1n=\1n h(s,\bar X(T))\1n+\1n\int_s^T\3n g\big(s,r,\bar X(r),
\psi(r,r,\bar X(r),
\Th(r,r,\bar X(r)),\Th_x(r,r,\bar X(r))),\bar Y(r),\bar Z(s,r)\big)dr\\
\ns\ds\qq\qq\qq\qq-\int_s^T \bar Z(s,r)dW(r),\qq s\in[\t,T].\ea\right.\ee
On the other hand, we have
\bel{bar Y*}\bar Y(r)=\Th(r,r,\bar X(r))\equiv V(r,\bar X(r)),\qq r\in[\t,T],\ee
and
\bel{bar Z*}\bar Z(s,r)=\Th_x(s,r,\bar X(r))\si(r,\bar X(r)),\qq(s,r)\in\D[\t,T].\ee
Thus,
\bel{bar Z**}\Th_x(r,s,\bar X(r))=\bar Z(s,r)\si(r,\bar X(r))^{-1},\qq(s,r)\in\D[\t,T].\ee
Consequently, the above closed-loop system can be written as follows.
\bel{FSDE-BSVIE*}\left\{\2n\ba{ll}
\ds\bar X(s)=\xi+\int_\t^sb\big(r,\bar X(r),\psi(r,r,\bar X(r),
\bar Y(r),\bar Z(r,r)\si(r,\bar X(r))^{-1})\big)dr+\int_\t^s\si(r,\bar X(r))dW(r),\\
\ns\ds\bar Y(s)\1n=\1n h(s,\bar X(T))\1n+\1n\int_s^T\3n g\big(s,r,\bar X(r),
\psi(r,r,\bar X(r),
\bar Y(r),\bar Z(r,r)\si(r,\bar X(r))^{-1}),\bar Y(r),\bar Z(s,r)\big)dr\\
\ns\ds\qq\qq\qq\qq-\int_s^T \bar Z(s,r)dW(r),\qq s\in[\t,T].\ea\right.\ee
If we denote
\bel{bar b,bar g}\ba{ll}
\ns\ds\bar b(r,x,y,\z)=b\big(r,x,\psi(r,r,x,y,\z\si(r,x)^{-1}\big),\qq(r,x,y,\z)\in[\t,T]\times\dbR^n
\times\dbR\times\dbR^{1\times n},\\
\ns\ds\bar g\big(s,r,x,y,z,\z)=g(s,r,x,\psi(r,r,x,y,\z\si(r,x)^{-1}),y,z\big),\\
\ns\ds\qq\qq\qq\qq\qq\qq(s,r,x,y,z,\z)\in\D[\t,T]
\times\dbR^n\times\dbR\times\dbR^{1\times n}\times\dbR^{1\times n}.\ea\ee
Then the above closed-loop system can further be written as follows.
\bel{FSDE-BSVIE**}\left\{\2n\ba{ll}
\ds\bar X(s)=\xi+\int_\t^s\bar b\big(r,\bar X(r),\bar Y(r),\bar Z(r,r)\big)dr+\int_\t^s\si(r,\bar X(r))dW(r),\\
\ns\ds\bar Y(s)\1n=\1n h(s,\bar X(T))\1n+\1n\int_s^T\3n\bar g\big(s,r,\bar X(r),
\bar Y(r),\bar Z(s,r),\bar Z(r,r)\big)dr-\int_s^T \bar Z(s,r)dW(r),\ea\right.\q s\in[\t,T].\ee
Unlike the BSVIEs studied in the literature, the above BSVIE contains the {\it diagonal values} $\bar Z(r,r)$ of $Z(\cd\,,\cd)$. To our best knowledge, this is the first time that such a BSVIE appears. Our above results show that the above coupled FSDE and BSVIE admits an adapted solution $(\bar X(\cd),\bar Y(\cd),\bar Z(\cd\,,\cd))$. Moreover, the representation \rf{bar Y*}--\rf{bar Z*} holds, with $\Th(\cd\,,\cd\,,\cd)$ being the classical solution to the equilibrium HJB equation \rf{Th-differential-form-si}. We point out that \rf{bar Y*} and \rf{bar Z**}, which represents the solution $\Th(\cd\,,\cd\,,\cd)$ to the equilibrium HJB equation \rf{Th-differential-form-si}, is a kind of {\it Feynman--Kac formula}.

\ms

The above naturally motivates us to investigate the following more general coupled FSDE and BSVIE:
\bel{FSDE-BSVIE-general}\left\{\2n\ba{ll}
\ds X(s)=\xi+\int_\t^s b\big(r,X(r),Y(r),Z(r,r)\big)dr+\int_\t^s\si\big(r,X(r),Y(r)\big)dW(r),\\
\ns\ds Y(s)\1n=\1n h(s,X(T))\1n+\1n\int_s^T\3n g\big(s,r,X(r),
Y(r),Z(s,r),Z(r,r)\big)dr-\int_s^TZ(s,r)dW(r),\q s\in[\t,T].\ea\right.\ee
The main feature is that the generator $g$ of the above BSVIE contains the diagonal value $Z(r,r)$. In the rest of this section, we will sketch some relevant results of the above coupled FSDE and BSVIE. More general detailed investigation of such BSVIEs will be carried out elsewhere.

\ms

Inspired by the results of previous sections, as well as the ideas from \cite{Ma-Protter-Yong 1994, Wang-Yong 2019}, we let $(t,s,x)\mapsto\Th(t,s,x)$ be $C^{0,1,2}(\D[0,T]\times\dbR^n)$. Applying It\^o's formula to the process $s\mapsto\Th(t,s,X(s))$, one obtains
%
%
%
%
%
%
\bel{Th-Th}\ba{ll}
\ns\ds\Th(t,T,X(T))-\Th(t,t,X(t))=\int_t^T\[\Th_s(t,s,X(s))+\Th_x(t,s,X(s))b(s,X(s),Y(s),Z(s,s))\\
\ns\ds\qq\qq\qq\qq\qq\qq\qq\qq+{1\over2}\tr\(\Th_{xx}(t,s,X(s))a(s,X(s),Y(s))\)\]ds\\
\ns\ds\qq\qq\qq\qq\qq\qq\qq\qq+\int_t^T\Th_x(t,s,X(s))\si(s,X(s),Y(s))dW(s),\ea\ee
where
$$a(s,x,y)=\si(s,x,y)\si(s,x,y)^\top,\qq(s,x,y)\in[0,T]\times\dbR^n\times\dbR.$$
Comparing \rf{Th-Th} with the BSVIE in \rf{FSDE-BSVIE-general}, we see that the following should be the right choice:
\bel{Th=h}\Th(t,T,X(T))=h(t,X(T)),\ee
\bel{Y,Z}Y(t)=\Th(t,t,X(t)),\qq Z(t,s)=\Th_x(t,s,X(s))\si(s,X(s),Y(s)),\ee
and
$$\ba{ll}
\ns\ds\Th_s(t,s,X(s))+\Th_x(t,s,X(s))b\big(s,X(s),Y(s),Z(s,s)\big)\\
\ns\ds\q+{1\over2}\tr\(\Th_{xx}(t,s,X(s))a(s,X(s),Y(s))\)+g\big(t,s,X(s),Y(s),Z(t,s),
Z(s,s)\big)=0.\ea$$
Thus, formally, we should have
$$\ba{ll}
\ns\ds\Th_s(t,s,X(s))+\Th_x(t,s,X(s))b\big(s,X(s),\Th(s,s,X(s)),\Th_x(s,s,X(s))\si(s,X(s),
\Th(s,s,X(s)))\big)\\
\ns\ds\qq+{1\over2}\tr\(\Th_{xx}(t,s,X(s))a\big(s,X(s),\Th(s,s,X(s))\big)\)\\
\ns\ds\qq+g\big(t,s,X(s),\Th(s,s,X(s)),\Th_x(t,s,X(s))\si(s,X(s),\Th(s,s,X(s))),\\
\ns\ds\qq\qq\qq\qq\qq\qq\Th_x(s,s,X(s))\si(s,X(s),\Th(s,s,X(s)))\big)=0.\ea$$
Hence, we have the following result.

\bt{theorem-coupled-SDE-BSVIE} \sl Let the following system admit a classical solution $\Th(\cd\,,\cd\,,\cd)$:
\bel{PDE}\left\{\2n\ba{ll}
\ds\Th_s(t,s,x)+{1\over2}\tr\[\Th_{xx}(t,s,x)a(s,x,\Th(s,s,x))\]\\
\ns\ds\q+\Th_x(t,s,x)b\big(s,x,\Th(s,s,x),
\Th_x(s,s,x)\si(s,x,\Th(s,s,x))\big)\\
\ns\ds\q+g\big(t,s,x,\Th(s,s,x),\Th_x(t,s,x)\si(s,x,\Th(s,s,x)),\Th_x(s,s,x)
\si(s,x,\Th(s,s,x))\big)=0,\\
\ns\ds\Th(t,T,x)=h(t,x),\qq(t,x)\in[0,T]\times\dbR^n,\ea\right.\ee
Let $X(\cd)\equiv X(\cd\,;\t,\xi)$ be the solution to the following FSDE:
$$\ba{ll}
\ns\ds X(s)=\xi+\int_\t^sb\big(r,X(r),\Th(r,r,X(r)),\Th_x(r,r,X(r))\si(r,X(r),
\Th(r,r,X(r)))\big)dr\\
\ns\ds\qq\qq+\int_\t^s\si\big(s,X(s),\Th(r,r,X(r))\big)dW(r),\qq s\in[\t,T].\ea$$
Then $(Y(\cd),Z(\cd\,,\cd))$ defined by \rf{Y,Z} is an adapted solution to the
BSVIE in \rf{FSDE-BSVIE-general}.

\et

When $\si(s,x,\th)$ is independent of $\th$, \autoref{theorem-PDE} provides a sufficient condition for the well-posedness of \rf{PDE}.
We now look at the uniqueness of the adapted solutions \rf{FSDE-BSVIE-general}. To this end, let us first introduce the following assumption.
\begin{taggedassumption}{(H6)}\label{ass:A2PDE} \rm
Let $d=n$.
There exist  maps $\mu:\D[0,T]\to\dbR$, $\n:\D[0,T]\to\dbR$, $g_0:[0,T]\times\dbR^n\times\dbR\times\dbR^{1\times n}\to\dbR$, $h_0:[0,T]\times\dbR^n\to\dbR$, $\a:[0,T]\to\dbR^n$
such that
\bel{A21}
\begin{aligned}
  &h(t,x)=\mu(t,T)h_0(T,x),\q \forall~(t,x)\in[0,T]\times\dbR^n,\\
  &\bar g(t,s,x,y,\z,z)=\n(t,s)g_0(s,x,y,\z)+z\a(s),\\
  &\qq\qq\qq \forall~(t,s,x,y,\z,z)\in\D[0,T]\times\dbR^n\times\dbR\times\dbR^{1\times n}\times\dbR^{1\times n}.
\end{aligned}
\ee
Suppose that the above maps are all bounded, have all required differentiability with bounded derivatives.
There exist continuous maps $\cM,\cN,\cK:[0,T]\to\dbR$ and $M:[0,T]^2\to\dbR$ such that
\bel{A22}
\begin{aligned}
  &  \cM(t)>0, \q \cM(t)\mu(t,T)-\cM(s)\mu(s,T)=M(t,s)\cN(T),\q t,s\in [0,T],\\
  &  \cM(t)\n(t,r)-\cM(s)\n(s,r)=M(t,s)\cK(r),\q t,s,r\in[0,T].
\end{aligned}
\ee
\end{taggedassumption}

Let us list some possible  functions $\mu(\cd,\cd)$ and $\n(\cd,\cd)$ satisfying (H6) as follows:
\begin{enumerate}[(i)]
\item {\it Heterogeneous discounting}: $\mu(t,T)=e^{-\l_1(T-t)}$, $\nu(t,r)=e^{-\l_2(r-t)}$ with $\l_1,\l_2>0$, $\l_1\ne\l_2$.
 We can let $\cM(t)=e^{-\l_1t}$,  $M(t,s)=e^{(\l_2-\l_1)t}-e^{(\l_2-\l_1)s}$, $\cN(T)=0$, $\cK(r)=e^{-\l_2r}$.
%
\item {\it Convex combination of two exponential discounting }: $\mu(t,T)=\a e^{-\l_1(T-t)}+(1-\a)e^{-\l_2(T-t)} $,
$\nu(t,r)=\a e^{-\l_1(r-t)}+(1-\a)e^{-\l_2(r-t)} $, with $\a\in(0,1)$, $\l_1,\l_2>0$, $\l_1\ne\l_2$.  We can let $\cM(t)=e^{-\l_1t}$,  $M(t,s)=e^{(\l_2-\l_1)t}-e^{(\l_2-\l_1)s}$, $\cN(T)=(1-\a)e^{-\l_2 T}$, $\cK(r)=(1-\a)e^{-\l_2r}$.
  %
%
%
\item {\it Quasi-exponential discounting}: $\mu(t,T)=\big(1+\a(T-t)\big)e^{-\l(T-t)}$,
$\nu(t,r)=\big(1+\a(r-t)\big)e^{-\l(r-t)}$, with $\a,\l>0$.  We can let $\cM(t)=e^{-\l t}$,  $M(t,s)=s-t$, $\cN(T)=\a e^{-\l T}$, $\cK(r)=\a e^{-\l r}$.
%
\end{enumerate}

Under (H6), the equation \rf{FSDE-BSVIE**} becomes
\bel{coupled-SDE-BSVIE-bar-mu-v}\left\{\begin{aligned}
\bar X(t) &=\xi+\int_\t^t\bar b(r,\bar X(r),\bar Y(r),\bar Z(r,r))dr+\int_\t^t\si(r,\bar X(r))dW(r),\q t\in[\t,T],\\
\bar Y(t) &=\mu(t,T)h_0(T,\bar X(T)) +\int_t^T \big[\n(t,r) g_0(r,\bar X(r),\bar Y(r),\bar Z(r,r))+\bar Z(t,r)\a(r)\big]dr\\
&\qq\qq\qq\qq\q  -\int_t^T\bar Z(t,r)dW(r),\q t\in [\t,T].
\end{aligned}\right.\ee
For the above coupled SDE and BSVIE, we have the following well-posedness result.

\begin{theorem}\label{theorem-uniqueness} \sl
Let {\rm (H5)} and {\rm (H6)} hold. Then the coupled SDE and BSVIE \rf{coupled-SDE-BSVIE-bar-mu-v} admits a unique adapted solution.
\end{theorem}

\begin{proof}
The existence of the adapted solution to \rf{coupled-SDE-BSVIE-bar-mu-v} is a combination of  \autoref{theorem-PDE} and \autoref{theorem-coupled-SDE-BSVIE}.
We only need to prove the uniqueness here.
Let $(\bar X_i(\cd),\bar Y_i(\cd),\bar Z_i(\cd,\cd))$; $i=1,2$ be two adapted solutions to \rf{coupled-SDE-BSVIE-bar-mu-v}. By the uniqueness of BSVIE, there exist uniquely $(y_i(\cd\,,\cd),z_i(\cd\,,\cd));i=1,2$ such that
\bel{coupled-SDE-BSVIE-bar-mu-v-proof}
\begin{aligned}
 y_i(t,s) &=\mu(t,T)h_0(T,\bar X_i(T))+\int_s^T\big[\n(t,r)g_0(r,r \bar X_i(r),\bar Y_i(r),\bar Z_i(r,r))+z_i(t,r)\a(r) \big]dr\\
&\qq\qq\qq\qq  -\int_s^T z_i(t,r)dW(r),\q (t,s)\in \D[\t,T],~i=1,2,
\end{aligned}\ee
and
\bel{l-z-Y-Z}
y_i(t,t)=\bar Y_i(t),\q z_i(t,s)=\bar Z_i(t,s),\q(t,s)\in\D[\t,T],~i=1,2.
\ee
For any $t\in[\t,T)$, multiplying the both side of BSVIE \rf{coupled-SDE-BSVIE-bar-mu-v-proof} by $\cM(t)$, we have
\begin{align*}
\cM(t)y_i(t,s) &=\cM(t)\mu(t,T)h_0(T,\bar X_i(T)) +\int_s^T \big[\cM(t)\n(t,r) g_0(r, \bar X_i(r),y_i(r,r), z_i(r,r))\\
&\qq\qq\qq+\cM(t) z_i(t,r)\a(r)\big]dr-\int_s^T\cM(t) z_i(t,r)dW(r),\q (t,s)\in \D[\t,T].
\end{align*}
Let
\bel{ti-Y-Z}
\wt X_i(t)=\bar X_i(t),\q \wt y_i(t,s)=\cM(t)y_i(t,s),\q \wt z_i(t,s)=\cM(t)z_i(t,s),\q (t,s)\in\D[\t,T],
\ee
then
\begin{align*}
\wt y_i(t,s) &=\cM(t)\mu(t,T)h_0(T,\wt X_i(T)) +\int_s^T \big[\cM(t)\n(t,r) g_0\big(r, \wt X_i(r),\cM^{-1}(r)\wt y_i(r,r), \cM^{-1}(r)\wt z_i(r,r)\big)\\
&\qq\qq\qq\qq +\wt z_i(t,r)\a(r)\big]dr -\int_s^T\wt z_i(t,r)dW(r),\q (t,s)\in \D[\t,T].
\end{align*}
For any $t'\in[\t,T)$, combining the above with \rf{A22}, $\big(\wt y_i(t,s)-\wt y_i(t',s),\wt z_i(t,s)-\wt z(t',s)\big);t\vee t'\les s\les T$ satisfies
\begin{align*}
\wt y_i(t,s)-\wt y_i(t',s) &=\big[\cM(t)\m(t,T)-\cM(t')\m(t',T)\big]h_0(T,\ti X_i(T)) \\
&\q+\int_s^T \Big[\big[\cM(t)v(t,r)-\cM(t')v(t',r)\big] g_0\big(r,\wt X_i(r),\cM^{-1}(r)\wt y_i(r,r), \cM^{-1}(r)\wt z_i(r,r)\big)\\
&\qq\qq +\big[\wt z_i(t,r)-\wt z_i(t',r)\big]\a(r)\Big]dr-\int_s^T\big[\wt z_i(t,r)-\wt z_i(t',r)\big]dW(r)\\
&=M(t,t')\cN(T)h_0(T,\ti X_i(T)) \\
&\q+\int_s^T \Big[M(t,t')\cK(r)g_0\big(r,\wt X_i(r),\cM^{-1}(r)\wt y_i(r,r), \cM^{-1}(r)\wt z_i(r,r)\big)\\
&\qq\qq +\big[\wt z_i(t,r)-\wt z_i(t',r)\big]\a(r)\Big]dr -\int_s^T\big[\wt z_i(t,r)-\wt z_i(t',r)\big]dW(r).
\end{align*}
Let $(\h Y_i(\cd),\h Z_i(\cd));i=1,2$ be the unique solution to the following BSDE:
\bel{Y-Z-hat}
\begin{aligned}
\h Y_i(s)&=\cN(T)h_0(T,\wt X_i(T)) \\
&\q+\int_s^T \Big[\cK(r)g_0\big(r, \wt X_i(r),\cM^{-1}(r)\wt y_i(r,r), \cM^{-1}(r)\wt z_i(r,r)\big)+\h Z_i(r)\a(r)\Big]dr \\
&\q-\int_s^T\h Z_i(r)dW(r),\q s\in [\t,T],
\end{aligned}\ee
respectively. Multiplying the both side of BSDE \rf{Y-Z-hat} by $M(t,t')$, $\big(M(t,t')\h Y_i(\cd),M(t,t')\h Z_i(\cd)\big)$ satisfies
\begin{align*}
M(t,t')\h Y_i(s)&=M(t,t')\cN(T)h_0(T,\wt X_i(T)) \\
&\q+\int_s^T \Big[M(t,t')\cK(r)g_0\big(r, \wt X_i(r),\cM^{-1}(r)\wt y_i(r,r), \cM^{-1}(r)\wt z_i(r,r)\big)+M(t,t')\h Z_i(r)\a(r)\Big]dr \\
&\q-\int_s^TM(t,t')\h Z_i(r)dW(r),\q s\in [t\vee t',T].
\end{align*}
For the fixed $t,t'$, by the uniqueness of the adapted solution to the above BSDE, we have
$$\wt y_i(t,s)-\wt y_i(t',s)=M(t,t')\h Y_i(s),\q\wt z_i(t,s)-\wt z_i(t',s)=M(t,t')\h Z_i(s),\q s\in [t\vee t',T].$$
In particular, the following holds by taking $t=s$ and $t'=\t$,
\bel{relationship}
\wt y_i(s,s)-\wt y_i(\t,s)=M(s,\t)\h Y_i(s),\q \wt z_i(s,s)-\wt z_i(\t,s)=M(s,\t)\h Z_i(s).\ee
Thus, $(\wt X_i(\cd),\wt y_i(\t,\cd),\h Y_i(\cd),\wt z_i(\t,\cd),\h Z_i(\cd))$ satisfies the following coupled FBSDEs:
$$\left\{\begin{aligned}
\wt X_i(s) &=\xi+\int_\t^s\bar b\big (r,\wt X_i(r),\cM^{-1}(r)[\wt y_i(\t,r)+M(r,\t)\h Y_i(r)],\cM^{-1}(r)[\wt z_i(\t,r)+M(r,\t)\h Z_i(r)]\big)dr\\
&\q+\int_\t^s\si(r,\wt X_i(r))dW(r),\q s\in[\t,T],\\
\wt y_i(\t,s) &=\cM(\t)\mu(\t,T)h_0(T,\wt X_i(T))\\
&\q+\int_s^T\Big [\cM(\t)v(\t,r) g_0\big(r, \wt X_i(r),\cM^{-1}(r)\big[\wt y_i(\t,r)+M(r,\t)\h Y_i(r)\big] , \cM^{-1}(r)\big[\wt z_i(\t,r)+M(r,\t)\h Z_i(r)\big]\big)\\
&\qq\qq+\wt z_i(\t,r)\a(r)\Big]dr -\int_s^T\wt z_i(\t,r)dW(r),\q s\in [\t,T],\\
\hat Y_i(s)&=\cN(T)h_0(T,\wt X_i(T)) \\
&\q+\int_s^T \Big[\cK(r)g_0\big(r,\wt X_i(r),\cM^{-1}(r)\big[\wt y_i(\t,r)+M(r,\t)\h Y_i(r)\big], \cM^{-1}(r)\big[\wt Z_i(\t,r)+M(r,\t)\h Z_i(r)\big]\big) \\
&\qq\qq +\h Z_i(r)\a(r)\Big]dr-\int_s^T\h Z_i(r)dW(r),\q s\in [\t,T].
\end{aligned}\right.
$$
By \cite[Theorem 4.1]{Ma-Protter-Yong 1994}, the above coupled FBSDEs admit a unique adapted solution.
It follows that
$$\big(\wt X_1(\cd),\wt y_1(\t,\cd),\h Y_1(\cd),\wt z_1(\t,\cd),\h Z_1(\cd)\big)=\big(\wt X_2(\cd),\wt y_2(\t,\cd),\h Y_2(\cd),\wt z_2(\t,\cd),\h Z_2(\cd)\big).$$
%
Combining this with \rf{relationship}, we have
$$
\wt y_1(s,s)=\wt y_2(s,s),\q \wt z_1(s,s)=\wt z_2(s,s), \q s\in[\t,T].
$$
By the definition \rf{ti-Y-Z} of $(\wt X_i(\cd),\wt y_i(\cd,\cd),\wt z_i(\cd,\cd))$  and the relationship \rf{l-z-Y-Z}, we have
\bel{X-Y(s,s)-Z(s,s)}
\begin{aligned}
 &\bar X_1(s)=\bar X_2(s),\q \bar Y_1(s)=y_1(s,s)=y_2(s,s)=\bar Y_2(s),\\
 &\bar Z_1(s,s)=z_1(s,s)=z_2(s,s)=\bar Z_2(s,s), \q s\in[\t,T].
\end{aligned}
\ee
Then $(\bar Y_i(\cd),\bar Z_i(\cd,\cd))$ satisfies the following BSVIE:
\bel{BSVIE-bar-Z(s,s)}
\begin{aligned}         %
 \bar Y_i(t) &=\mu(t,T)h_0(T,\bar X(T)) +\int_t^T \big[v(t,r) g_0(r, \bar X(r),\bar Y(r),\bar Z(r,r))+\bar Z_i(t,r)\a(r)\big]dr\\
&\qq\qq\qq\qq  -\int_s^T \bar Z_i(t,r)dW(r),\q (t,s)\in \D[\t,T],
\end{aligned}\ee
with $\bar X(s)\deq \bar X_1(s)\equiv \bar X_2(s), \bar Y(s)\deq \bar Y_1(s)\equiv \bar Y_2(s), \bar Z(s,s)\deq\bar Z_1(s,s)\equiv\bar Z_2(s,s); s\in[\t,T]$.
Note that \rf{BSVIE-bar-Z(s,s)} is a classical BSVIE (without $\bar Z_i(r,r)$).
By  \autoref{lmm:well-posedness-SDE}, we have
$$
(\bar Y_1(\cd),\bar Z_1(\cd,\cd))= (\bar Y_2(\cd),\bar Z_2(\cd,\cd)).
$$
Combining the above with \rf{X-Y(s,s)-Z(s,s)}, the uniqueness of the adapted solution to the coupled FSDE and BSVIE \rf{coupled-SDE-BSVIE-bar-mu-v} is obtained.
\end{proof}
\begin{remark}\rm
Under (H5)--(H6),  \autoref{theorem-uniqueness} establishes the well-posedness of the coupled SDE and BSVIE \rf{coupled-SDE-BSVIE-bar-mu-v},
which is relevant to several important recursive optimal control problems with nonexponential discounting.
The more general case \rf{FSDE-BSVIE-general} is still open. We hope to explore that in our future publications.
\end{remark}

\section{Concluding Remarks}\label{remarks}

In this section, we are making some remarks to conclude this paper.

\ms

First of all, for recursive cost functional with nonexponential discounting, should one use parameterized BSDEs as in \cite{Yong 2012,Wei-Yong-Yu 2017} or use BSVIE as in the current paper? For a stochastic optimal control problem, a recursive cost functional, with exponential discounting, can be described by the adapted solution to a BSDE. When the discounting is nonexponential, and/or the running cost rate and the terminal cost are initial time dependent, then the recursive cost functional had better to use a BSVIE, instead of a parameterized BSDE (as in \cite{Yong 2012,Wei-Yong-Yu 2017}). To be convincing, let us present a brief argument on that.

\ms

For any initial pair $(t,\xi)\in\sD$ and control $u(\cd)\in\sU[t,T]$, let $X(\cd)$ be the corresponding state process. Motivated by the nonexponential discounting,
one may consider the following cost functional
\bel{cost-none1}
 \h J(t,\xi;u(\cd))=\dbE_t\[h(t,X(T))+\int_t^T g(t,s,X(s),u(s))ds\]
\ee
with suitable functions $h(\cd)$ and $g(\cd)$.
If we let
\bel{cost-none-y}
 y(t,s)=\dbE_s\[h(t,X(T))+\int_s^Tg(t,r,X(r),u(r))dr\],\qq s\in[t,T],\ee
then for some $ z(\cd,\cd)$, the pair $( y(\cd,\cd), z(\cd,\cd))$ is the unique adapted solution to the following BSDE (with parameter $t$):
\bel{BSDE-noyz}
y(t,s)=h(t,X(T))+\int_s^T g(t,r,X(r),u(r))dr-\int_s^T z(t,r)dW(r),\q s\in[t,T],\ee
and
$$\h J(t,\xi;u(\cd))=y(t,t).$$
Following the above idea and inspired by \cite{Duffie-Epstein 1992,El Karoui-Peng-Quenez 1997,Wei-Yong-Yu 2017}, to construct the recursive cost functional for the state-control pair $(X(\cd),u(\cd))$, one might intuitively consider the following BSDE (with parameter $t$)
\bel{BSDE-yz}
y(t,s)=h(t,X(T))+\int_s^T g(t,r,X(r),u(r),y(t,r),z(t,r))dr-\int_s^T z(t,r)dW(r),\q r\in[t,T],\ee
and define the cost functional by
\bel{h J^R}\h J(t,\xi;u(\cd))=y(t,t).\ee
But, taking $s=t$ in \rf{BSDE-noyz}, we obtain
\bel{BSDE-yz1}
y(t,t)=h(t,X(T))+\int_t^T g(t,r,X(r),u(r),y(t,r),z(t,r))dr-\int_t^T z(t,r)dW(r),\q t\in[0,T],\ee
which is not an equation for the process $t\mapsto y(t,t)$ since $y(t,s)$ appears on the right-hand side of the above. Further, since $y(s,s)\neq y(t,s)$ in general, one finds that the dependence of current cost functional value $y(t,t)$ on the future cost functional values $\h J(s,X(s);u(\cd))=y(s,s);t< s\les T$ seems not to be valid. From this viewpoint, $y(t,t)$ defined above seems not to be a good candidate of the recursive cost functional with nonexponential discounting.

\ms

However, the above parameterized BSDE suggests us try to consider the following BSVIE:
\bel{BSVIE-yz}
Y(t)=h(t,X(T))+\int_t^T g(t,r,X(r),u(r),Y(r),Z(t,r))dr-\int_t^T Z(t,r)dW(r),\ee
and define the cost functional by
\bel{def-J} J(t,\xi;u(\cd))=Y(t),\ee
as we have done in this paper.
We know that under proper condition, the above BSVIE admits a unique adapted solution $(Y(\cd),Z(\cd,\cd))$ and process $Y(\cd)$ is the most natural candidate for the recursive cost functional with nonexponential discounting in the following sense: The current cost functional value $J(t,\xi;u(\cd))=Y(t)$ really depends on the cost functional values $J(r,X(r);u(\cd))=Y(r)$ for $r\in[t,T]$, through a BSIVE. Furthermore, such a recursive cost functional is time-consistent itself, in the following sense: The future value of the cost functional predicted/calculated today will match the value of the cost functional when that specific future time moment arrives. Some more detailed derivation and discussion can be found in \cite{Wang-Sun-Yong 2019}. In a word, when we consider the stochastic optimal control problems with recursive cost functional having generalized (nonexponential) discounting, BSVIE description should be a more proper choice than the parameterized BSDE.

\ms

Now, let use make a direct comparison between the equilibrium HJB equations resulting from the approach of \cite{Wei-Yong-Yu 2017} and the one of this paper. For convenience, we only consider the case that $\si$ is independent of control $u$. As in \cite{Wei-Yong-Yu 2017}, if the recursive cost functional is taken to be \rf{h J^R}, then the equilibrium HJB equation reads
\bel{Th-differential-form-si*}
\left\{\begin{aligned}
& \Th_s(t,s,x)+{1\over2}\tr[\Th_{xx}(t,s,x)a(s,x)]+\Th_{x}(t,s,x) b\big(s,x,\psi(s,s,x,\Th(s,s,x),\Th_x(s,s,x))\big)\\
&\qq +g\big(t,s,x,\psi(s,s,x,\Th(s,s,x),\Th_x(s,s,x)),{\color{red}\Th(t,s,x)},\Th_x(t,s,x) \si(s,x)\big)=0,\\
&\qq\qq\qq\qq\qq\qq\qq\qq\qq (t,s,x)\in \D[0,T]\times\dbR^n,\\
& \Th(t,T,x)=h(t,x),\q (t,x)\in[0,T]\times\dbR^n.\end{aligned}\right.\ee
Comparing the above with \rf{Th-differential-form-si}, we see that $\Th(t,s,x)$ in the above is replaced by $\Th(s,s,x)$ in \rf{Th-differential-form-si}. This is the main consequence of using BSVIE instead of parameterized BSDE. As we explained above, such a replacement makes the problem more natural. On the other hand, from the mathematical viewpoint, since $\Th(s,s,x)$ has been appeared in $\psi(\cd)$, regardless the above replacement, therefore, replacing $\Th(t,s,x)$ by $\Th(s,s,x)$ mathematically reduces the complexity of the equation.

\ms

Next, in the current paper, by using the idea inspired by multi-person differential games (\cite{Yong 2012, Wei-Yong-Yu 2017}) and representation of adapted solutions to BSVIEs (\cite{Yong 2017b,Wang-Yong 2019}), we obtain the equilibrium strategy for Problem (N), which is time-consistent, locally near optimal, and it is determined by the solution to an equilibrium HJB equation. We have seen that the equilibrium HJB equation in this paper is an interesting modification of that found in \cite{Wei-Yong-Yu 2017}.

\ms

Further, as a byproduct, in obtaining a Feynman-Kac type formula for the equilibrium HJB equation, we introduce a new class of BSVIEs for which the diagonal value $Z(s,s)$ of process $Z(\cd\,,\cd)$ appears. For such kind of equations, some very special cases have been studied and the general case is left widely open. Actually, our introduction in this paper initiates the research for such kind of BSVIEs. Some relevant ideas and results for the so-called extended BSVIEs can be found in \cite{Wang 2019}.

\ms

Finally, we provide a partial list of the widely open questions concerning our Problem (N):

\ms

$\bullet$ Solvability of general equilibrium HJB equation \rf{Th-differential-form-rewrite}, with non-degenerate and bounded diffusion, i.e.,
$$\l_0I\les\si(t,x,u)\si(t,x,u)^\top\les\l_1I.$$

$\bullet$ When $\si(t,x,u)$ is degenerate (and it is independent of $u$), is it possible to use viscosity solution to identify/characterize the equilibrium value function?

\ms

$\bullet$ If the map $\psi$ defined by \rf{define-psi} is not regular enough, or not unique, or even does not exist, what one can do for Problem (N)?
\ms


\begin{thebibliography}{90}
\addtolength{\itemsep}{-1.0ex}



\bibitem{Agram-Oksendal 2015} N.~Agram and B.~{\O}ksendal,
\it Malliavin calculus and optimal control of stchastic Volterra equations,
\sl J. Optim. Theory Appl.,
\rm {\bf167} (2015), 1070--1094.

\bibitem{Aman-N'Zi 2005} A.~Aman and M.~N'Zi,
\it Backward stochastic nonlinear Volterra integral equation with local Lipschitz drift,
\sl Probab. Math. Statist.,
\rm {\bf25} (2005), 105--127.

\bibitem{Anh-Grecksch-Yong 2011} V.~V.~Anh, W.~Grecksch, and J.~Yong,
\it Regularity of backward stochastic Volterra integral equations in Hilbert spaces,
\sl Stoch. Anal. Appl.,
\rm {\bf 29} (2011), 146--168.



\bibitem{Bender-Pokalyuk 2013} C.~Bender and S.~Pokalyuk,
\it Discretization of backward stochastic Volterra integral equations,
\sl Recent Developments in Computational Finance, Interdiscip. Math. Sci., 14,
\rm World Sci. Publ., Hackensack, NJ, 2013, 245--278.

\bibitem{Bjork-Khapko-Murgoci 2017} T.~Bj\"{o}rk, M.~Khapko, and A.~Murgoci,
\it On time-inconsistent stochastic control in continuous time,
\sl Finance Stoch.,
\rm {\bf 21} (2017), 331--360.

\bibitem{Bjork-Murgoci 2014}  T.~Bj\"{o}rk and A.~Murgoci,
\it A theory of Markovian time-inconsistent stochastic control in discrete time,
\sl Finance Stoch.,
\rm {\bf 18} (2014), 545--592.

\bibitem{Bjork-Murgoci-Zhou 2014}  T.~Bj\"{o}rk, A.~Murgoci, and X.~Y.~ Zhou,
\it Mean-variance portfolio optimization with state-dependent risk aversion,
\sl\sl Math. Finance,
\rm {\bf 24} (2014), 1--24.


\bibitem{Di Persio 2014} L.~Di Persio,
\it Backward stochastic Volterra integral equation approach to stochastic differential utility,
\sl Int. Electron. J. Pure Appl. Math.,
\rm {\bf8} (2014), 11--15.

\bibitem{Djordjevic-Jankovic 2013} J.~Djordjevi\'{c} and S.~Jankovi\'{c},
\it On a class of backward stochastic Volterra integral equations,
\sl Appl. Math. Lett.,
\rm {\bf26} (2013), 1192--1197.

\bibitem{Djordjevic-Jankovic 2015} J.~Djordjevi\'{c} and S.~Jankovi\'{c},
\it Backward stochastic Volterra integral equations with additive perturbations,
\sl Appl. Math. Comput.,
\rm {\bf 265} (2015), 903--910.

\bibitem{Duffie-Epstein 1992} D.~Duffie and L.~G.~Epstein,
\it Stochastic differential utility,
\sl Econometrica,
\rm {\bf 60} (1992), 353--394.

\bibitem{Duffie--Epstein 1992-1} D.~Duffie and L.~G.~Epstein,
\it Asset pricing with stochastic differential utility,
\sl The Review of Financial Studies, \rm {\bf 5} (1992), 411--436.


\bibitem{El Karoui-Peng-Quenez 1997} N.~El.~Karoui, S.~Peng, and M.~C.~Quenez,
\it Backward stochastic differential equations in finance,
\sl Math. Finance,
\rm {\bf7} (1997), 1--71.

\bibitem{Ekeland 2006} I. Ekeland and A. Lazrak,
\it Being Serious about Non-Commitment: Subgame Perfect Equilibrium in Continuoue Time,
\sl preprint,
\rm 2006, https://arxiv.org/abs/math/0604264.

\bibitem{Ekeland 2010} I. Ekeland and A. Lazrak,
\it The golden rule when preferences are time inconsistent,
\sl Math. Financ. Econ.,
\rm {\bf 4} (2010),  29--55.

\bibitem{Ekeland 2008} I. Ekeland and T. A. Pirvu,
\it Investment and consumption without commitment,
\sl Math. Financ. Econ.,
\rm {\bf 2} (2008),  57--86.


\bibitem{Hu 2018} Y.~Hu and B.~{\O}ksendal,
\it  Linear Volterra backward stochastic integral equations,
\sl Stochastic Process. Appl.,
\rm {\bf 129} (2019), 626--633.

\bibitem{Hu-Jin-Zhou 2012} Y.~Hu, H.~Jin, and X.~Y.~Zhou,
\it Time-inconsistent stochastic linear--quadratic control,
\sl  SIAM J. Control Optim.,
\rm {\bf 50} (2012), 1548--1572.

\bibitem{Hu-Jin-Zhou 2017} Y.~Hu, H.~Jin, and X.~Y.~Zhou,
\it Time-inconsistent stochastic linear-quadratic control: characterization and uniqueness of equilibrium,
\sl  SIAM J. Control Optim.,
\rm {\bf 55} (2017), 1261--1279.

\bibitem{Kromer-Overbeck 2017} E.~Kromer and L.~Overbeck,
\it Differentiability of BSVIEs and dynamical capital allocations,
\sl Int. J. Theor. Appl. Finance,
\rm {\bf20} (2017), No.07, 1750047.


\bibitem{Lazrak 2004} A.~Lazrak,
\it Generalized stochastic differential utility and preference for information,
\sl Ann. Appl. Probab.,
\rm {\bf 14} (2004), 2149--2175.

\bibitem{Lazrak-Quenez 2003} A.~Lazrak and M.~C.~Quenez,
\it A generalized stochastic differential utility,
\sl Math. Oper. Res.,
\rm {\bf 28} (2003), 154--180.

\bibitem{Lin 2002} J.~Lin,
\it Adapted solution of a backward stochastic nonlinear Volterra integral equation,
\sl Stoch. Anal. Appl.,
\rm {\bf 20} (2002), 165--183.


\bibitem{Ma-Protter-Yong 1994} J.~Ma, P.~ Protter, and J.~Yong,
\it Solving forward-backward stochastic differential equations explicitly --- four step scheme, \sl Probab. Theory Related Fields,
\rm {\bf 98} (1994), 339--359.

\bibitem{Ma-Yong 1999} J.~Ma and J.~Yong,
\it Forward-Backward Stochastic Differential Equations and Their Applications,
\sl Lecture Notes in Mathematics, {\bf1702},
\rm Springer-Verlag, Berlin, 1999.

\bibitem{Marin-Solano 2010}J. Marin-Solano and J. Navas,
\it Consumption and portfolio rules for time-inconsistent investors,
\sl Euro. J. Oper. Res.,
\rm {\bf201} (2010),  860--872.

\bibitem{Marin-Solano 2011} J. Marin-Solano and E. V. Shevkoplyas,
\it Non-constant discounting and differential games with random time horizon,
\sl Automatica,
\rm {\bf47} (2011), 2626--2638.

\bibitem{Mei-Yong 2019} H.~Mei and J.~Yong, \it Equilibrium strategies for time-inconsistent stochastic switching systems,
\sl ESAIM: COCV,
\rm {\bf25} (2019), 64.

\bibitem{Overbeck-Roder p} L.~Overbeck and J.~A.~L.~R\"oder,
\it Path-dependent backward stochastic Volterra integral equations with jumps, differentiability and duality principle,
\rm Probab. Uncertain. Quant. Risk, {\bf3} (2018), 4.


\bibitem{Pardoux-Peng 1990} E.~Pardoux and S.~Peng,
\it Adapted solution of a backward stochastic differential equation,
\sl Systems  Control Lett.,
\rm {\bf14} (1990), 55--61.

\bibitem{Pollak 1968} R.~A.~Pollak,
\it Consistent planning,
\sl Rev. Econ. Stud.,
\rm {\bf 35} (1968), 185--199.

\bibitem{Ren 2010} Y.~Ren,
\it On solutions of backward stochastic Volterra integral equations with jumps in Hilbert spaces,
\sl J. Optim. Theory Appl.,
\rm {\bf 144} (2010), 319--333.


\bibitem{Shi-Wang-Yong 2013} Y.~Shi, T.~Wang, and J.~Yong,
\it Mean-field backward stochastic Volterra integral equations,
\sl Discrete Contin. Dyn. Syst. Ser. B,
\rm {\bf 18} (2013), 1929--1967.


\bibitem{Shi-Wang-Yong 2015} Y.~Shi, T.~Wang, and J.~Yong,
\it Optimal control problems of forward-backward stochastic Volterra integral equations,
\sl Math. Control Relat. Fields,
\rm {\bf5} (2015), 613--649.

\bibitem{Wang 2019} H.~Wang, \it Extended backward stochastic Volterra integral equations, quasilinear parabolic equations, and Feynman-Kac formula, \rm arXiv:1908.07168v1 [math.PR] 20 Aug 2019.

\bibitem{Wang-Sun-Yong 2018} H.~Wang, J.~Sun, and J.~Yong,
\it Quadratic backward stochastic Volterra integral equations,
\sl arXiv preprint,
\rm arXiv:1810.10149v1 [math.PR] 24 Oct 2018.

\bibitem{Wang-Sun-Yong 2019} H.~Wang, J.~Sun and J.~Yong,
\it Recursive utility processes, dynamic risk measures and quadratic backward stochastic Volterra integral equations,
\rm preprint. 

\bibitem{Wang 2018} T.~Wang,
\it Linear quadratic control problems of stochastic Volterra integral equations,
\sl ESAIM: COCV,
\rm {\bf24} (2018), 1849--1879.


\bibitem{Wang-Yong 2015} T.~Wang and J.~Yong,
\it Comparison theorems for some backward stochastic Volterra integral equations,
\sl Stochastic Process. Appl.,
\rm {\bf125} (2015), 1756--1798.

\bibitem{Wang-Yong 2019} T.~Wang and J.~Yong,
\it Backward stochastic Volterra integral equations---representation of adapted solutions,
\sl Stochastic Process. Appl.,
\rm {\bf129} (2019), 4926--4964.

\bibitem{Wang-Zhang 2017} T.~Wang and H.~Zhang,
\it Optimal control problems of forward-backward stochastic Volterra integral equations with closed control regions,
\sl SIAM J. Control Optim.,
\rm {\bf55} (2017), 2574--2602.

\bibitem{Wang-Zhang 2007} Z.~Wang and X.~Zhang,
\it Non-Lipschitz backward stochastic Volterra type equations with jumps,
\sl Stoch. Dyn.,
\rm {\bf7} (2007), 479--496.




\bibitem{Wei-Yong-Yu 2017} Q.~Wei, J.~Yong, and Z.~Yu,
\it Time-inconsistent recursive stochastic optimal control problems,
\sl SIAM J. Control Optim.,
\rm {\bf 55} (2017), 4156--4201.


\bibitem{Yan-Yong 2019} W.~Yan and J.~Yong, \it Time-inconsistent optimal control problems and related issues, \sl Modeling, Stochastic Control, Optimization, and Applications, \rm Edited by G.~Yin and Q. Zhang, IMA Volumes in Mathematics and Its Applications, Vol. 164, Springer, 2019, 533--569.


\bibitem{Yong 2008} J.~Yong,
\it Well-posedness and regularity of backward stochastic Volterra integral equations,
\sl Probab. Theory Related Fields,
\rm {\bf 142} (2008), 21--77.

\bibitem{Yong 2012} J.~Yong, \it Time-inconsistent optimal control problems and the equilibrium HJB equation, \sl Math. Control Relat. Fields, \rm {\bf2} (2012), 271--329.

\bibitem{Yong 2014} J.~Yong, \it Time-inconsistent optimal control problems, in Proceedings of 2014 ICM, Section 16. Control Theory and Optimization, \rm (2014), 947--969.

\bibitem{Yong 2017} J.~Yong, \it Linear-quadratic optimal control problems for mean-field stochastic differential equations--time-consistent solutions, \sl Trans. Amer. Math. Soc., \rm {\bf 369} (2017), 5467--5523.

\bibitem{Yong 2017b} J.~Yong, \it Representation of adapted solutions to backward stochastic Volterra integral equations, \sl Scientia Sinica Mathematica, \rm {\bf 47} (2017), 1--12 (in Chinese).

\bibitem{Yong-Zhou 1999} J.~Yong and X.~Y.~Zhou,
\it Stochastic Control: Hamiltonian Systems and HJB Equations,
\rm Springer-Verlag, New York, 1999.

\bibitem{Zhou-1998} X.~Y.~Zhou, \it Stochastic near-opyimal controls: necessary and sufficient conditions for near-optimality, \sl SIAM J. Control Optim., \rm {\bf36} (1998), 929--947.

\end{thebibliography}
\end{document}